\documentclass{article}
\usepackage{amssymb}
\usepackage{amsmath}
\usepackage{mathtools}
\usepackage{csquotes}
\usepackage[amsmath,thmmarks]{ntheorem}
\usepackage{microtype}
\usepackage{xfrac}
\usepackage{tikz}
\usetikzlibrary{arrows,matrix}

\definecolor{cobalt}{rgb}{0.0,0.28,0.67}
\usepackage[colorlinks,urlcolor=cobalt,citecolor=cobalt,linkcolor=cobalt,pdftex]{hyperref}
\usepackage{cleveref}

\newtheorem{theorem}{Theorem}[section]
\newtheorem{lemma}[theorem]{Lemma}
\newtheorem{proposition}[theorem]{Proposition}
\newtheorem{corollary}[theorem]{Corollary}
\newtheorem{remark}[theorem]{Remark}
\newtheorem{definition}[theorem]{Definition}

\theoremstyle{nonumberplain}
\theoremheaderfont{\normalfont\itshape}
\theoremseparator{.}
\theorembodyfont{\upshape}
\theoremsymbol{\ensuremath{\Box}} % or \ensuremath{\Box}
\newtheorem{proof}{Proof}

\numberwithin{equation}{section}
\crefname{equation}{}{}

\DeclareMathOperator{\divergence}{div}
\DeclareMathOperator*{\argmin}{arg\,min}
\DeclareMathOperator*{\minimize}{minimize}
\DeclarePairedDelimiter\abs{\lvert}{\rvert}
\DeclarePairedDelimiter\norm{\lVert}{\rVert}

\DeclareMathOperator{\sfdivergence}{\mathsf{div}}
\newcommand{\sV}{\mathsf{V}}
\newcommand{\sv}{\mathsf{v}}
\newcommand{\sE}{\mathsf{E}}
\newcommand{\se}{\mathsf{e}}
\newcommand{\sW}{\mathsf{W}}
\newcommand{\sw}{\mathsf{w}}
\newcommand{\su}{\mathsf{u}}
\newcommand{\sff}{\mathsf{f}}
\newcommand{\sH}{\mathsf{H}}
\newcommand{\si}{\mathsf{i}}
\newcommand{\sj}{\mathsf{j}}

\newcommand{\sJ}{\mathsf{J}}
\newcommand{\sB}{\mathsf{B}}

\title{On the ROF Model in Rectilinear Anisotropy: Piecewise Constant Approximation and Universal Minimality}
\author{
	Clemens Kirisits\thanks{Faculty of Mathematics, University of Vienna, Oskar-Morgenstern-Platz 1, A-1090 Vienna, Austria; Christian Doppler Laboratory for Mathematical Modeling and Simulation of Next-Generation Ultrasound Devices (MaMSi), Oskar-Morgenstern-Platz 1, A-1090 Vienna, Austria\\ \texttt{clemens.kirisits@univie.ac.at}}
	\and
	Eric Setterqvist\thanks{Santa Anna IT Research Institute, SE-58183 Linköping, Sweden\\ \texttt{eric.setterqvist@santa-anna.se}}
	}
	
\begin{document}

\maketitle

\begin{abstract}
We prove that the $L^2$ distance between the minimizer of the $\ell^1$-anisotropic Rudin-Osher-Fatemi (ROF) functional and its minimizer over the space of piecewise constant functions on a rectilinear grid is $\mathcal{O}(h^{\sfrac12 - \sfrac{q'}{2q}})$, where $h$ is the grid's mesh size and the datum belongs to $L^q$, $q \ge 2$. These convergence rates are valid in any dimension $d\ge 1$. However, in dimension $d = 1$ they can be further improved to $\mathcal{O}(h^{\sfrac12 - \sfrac{1}{2q}})$. To establish the error bounds, $L^q$ estimates of the ROF minimizer in terms of the datum are critical. Such estimates are particular cases of a universal minimality property of the ROF minimizer derived in the second part of the paper. There it is shown, in both the finite-dimensional and infinite-dimensional settings, that the minimizer simultaneously minimizes a broad class of convex functionals over a neighbourhood of the datum arising in the convex dual of the ROF problem. This extends previous results of similar type about taut strings and the ROF problem.
\end{abstract}



\section{Introduction}\label{sec:intro}
Total variation regularization has become a standard in image processing and inverse problems, valued for its capability to suppress noise without losing edge sharpness. A starting point for this development was the publication of the Rudin-Osher-Fatemi (ROF) denoising model \cite{Rudin1}, which consists in minimizing the total variation balanced by a squared $L^2$ distance from the noisy datum.

\paragraph{Rectilinear anisotropy} This article is about an anisotropic variant of the ROF model. Specifically, given $f \in L^2(\Omega)$, we consider the problem 
\begin{align}\label{eq:rof-intro}
	\minimize_{u \in L^2(\Omega)\cap BV(\Omega)} \quad  \frac{1}{2}\left\| f - u \right\|^2_{L^2(\Omega) }+ \alpha J(u), \tag{P}
\end{align}
where $\Omega = \prod_{i=1}^d (a_i,b_i)$ is a bounded hyperrectangle of dimension $d\ge 1$, $\alpha > 0$ is a parameter controlling the degree of regularization and
\begin{align}\label{eq:J}
	J(u) = \sup  \left\{ \int_\Omega u \divergence H \, dx \,\middle|\, H \in C_c^\infty \big(\Omega, \mathbb{R}^d\big), \abs{H(x)}_\infty \le 1 \text{ for } x \in \Omega \right\}
\end{align}
is the $\ell^1$-anisotropic total variation of $u$. While it lacks the rotational symmetry of its isotropic counterpart, where  the Euclidean norm $\abs{\cdot}_2$ replaces the maximum norm $\abs{\cdot}_\infty$ in \cref{eq:J}, $J$ is a natural choice if one aims to exactly reconstruct rectilinear shapes, that is, shapes whose boundaries are aligned with the coordinate axes. In fact, the following result was shown in \cite{KirSchSet19,Lasica1}.
\begin{theorem}\label{thm:fpcrupcr-intro}
If $f$ is piecewise constant on a rectilinear grid, then the solution of \cref{eq:rof-intro} is piecewise constant on the same grid.
\end{theorem}
See \cref{fig:denoised-qr} for an illustration of this statement. Applications where this rectilinear bias is exploited can be found, for example, in \cite{BerBurDroNem06,choksi2011,SanOzkRomGok18,Setzer1}. Moreover, in the context of numerical analysis of conservation laws the total variation is frequently defined in the anisotropic fashion \cref{eq:J}. See, for instance, \cite{CraMaj80,HolRis11,LeV92}. We also remark that close connections between $\ell^1$-anisotropic total variation regularization and wavelet denoising, in particular Haar thresholding, have been revealed in e.g.\ \cite{CohDeVPetXu99}.

\begin{figure}
	\centering
	\includegraphics[width=.3\textwidth]{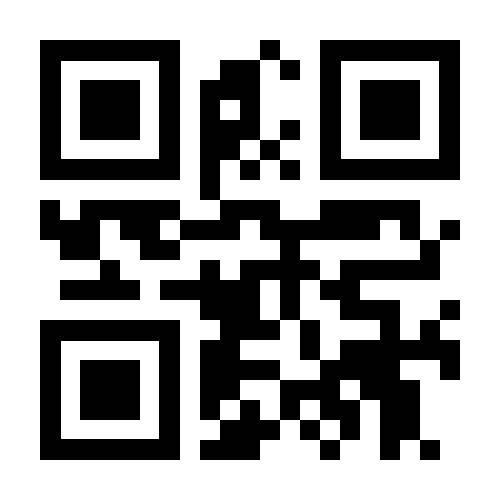} \quad
	\includegraphics[width=.3\textwidth]{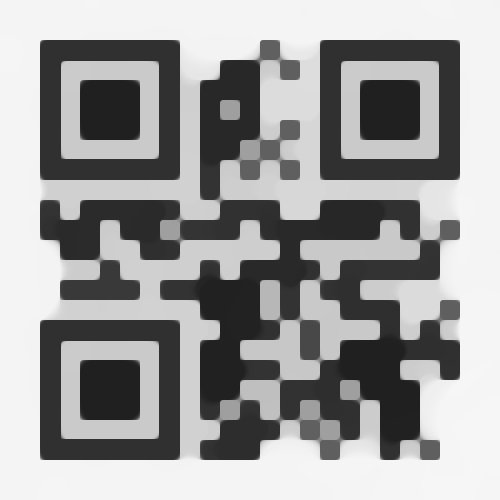} \quad
	\includegraphics[width=.3\textwidth]{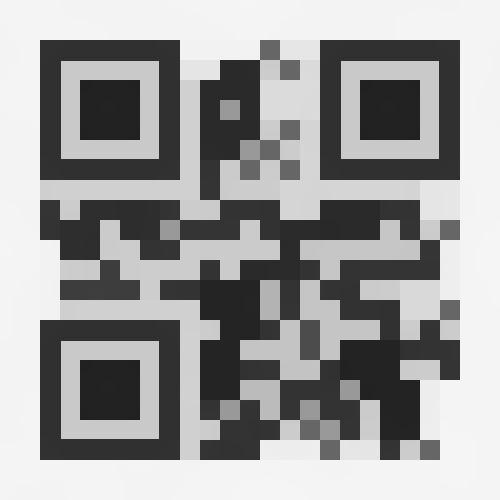}
	\caption{Left: QR code $f$. Mid: Minimizer of the isotropic ROF for datum $f$. Right: Minimizer of the $\ell^1$-anisotropic ROF for datum $f$. The mid and right images were obtained using the algorithm from \cite{ChaPoc11}.} \label{fig:denoised-qr}
\end{figure}

\paragraph{PCR approximation} One of the central themes of the present article is the finite-dimensional approximation of \cref{eq:rof-intro} using functions that are piecewise constant on rectilinear grids. Following \cite{Lasica1} we refer to these functions as $PCR$. Several intrinsic properties of $PCR$ functions make them a natural choice for approximating \cref{eq:rof-intro}.  One of these properties is stated in \cref{thm:fpcrupcr-intro}. Another one is the fact that for every $u \in L^p(\Omega ) \cap BV(\Omega)$, $1 \le p < \infty$, there is a sequence of $PCR$ functions such that $u_n \to u$ in $L^p(\Omega)$ and $J(u_n) \to J(u)$, see \cite[Thm.~3.4]{CasKunPol99}. Furthermore, if $u,f$ are $PCR$ with values $u_i,f_i$ on the $i$-th element, then the functional in \cref{eq:rof-intro} can be written as
\begin{align} \label{eq:discrete-energy-intro}
	\sum_i w_i \abs{u_i-f_i}^2 + \sum_{i,j} w_{ij} \abs{u_i-u_j}
\end{align}
for suitable weights $w_i,w_{ij}\ge 0$. Very efficient --- e.g.\ graph cut --- algorithms are available for the minimization of this type of discrete energy. See, for instance, \cite{BoyKol04,ChaDar09,GolYin09,Hoc01}.

\paragraph{Error bounds for PCR approximations of \cref{eq:rof-intro}} In \cite{CasKunPol99,KeeFit97} $PCR$ functions have been used to approximate $J$-regularized solutions of inverse problems. Our first result, \cref{thm:approximationerror}, extends these qualitative results by quantitative error estimates for problem \cref{eq:rof-intro}. Denote by $u_f$ the solution of \cref{eq:rof-intro} and by $u_{A_G f}$ the solution of \cref{eq:rof-intro} where $f$ is replaced by its projection $A_Gf$ onto the space of $PCR$ functions on a given grid $G$. Note that $u_{A_G f}$ is $PCR$ due to \cref{thm:fpcrupcr-intro} and can be found by minimizing a discrete energy of the form \cref{eq:discrete-energy-intro}. \Cref{thm:approximationerror} implies that for every $f \in L^q(\Omega)$, $2 \le q \le \infty$, 
\begin{align}\label{eq:l2rate}
	\norm[\big]{ u_{A_{G}f} - u_f}_{L^2(\Omega)} = \mathcal{O}\big(h^{\frac12 - \frac{q'}{2q}}\big)
\end{align}
as well as
\begin{align}\label{eq:tvrate}
	\abs[\big]{ J(u_{A_{G}f}) - J(u_f)} = \mathcal{O}\big(h^{\frac12 - \frac{q'}{2q}}\big),
\end{align}
where $h$ is the longest side length of all hyperrectangles in the partition of $\Omega$ and $q'$ is the Hölder conjugate of $q$. We note that the maximal rate of $h^{\sfrac12}$ (corresponding to $q=\infty$) for the $L^2$ approximation error is, in fact, optimal for $f \in L^\infty(\Omega) \cap BV(\Omega)$.

The rates \cref{eq:l2rate} and \cref{eq:tvrate} are valid in every dimension $d\ge 1$. In dimension $d=1$, however, where there is no anisotropy and \cref{eq:rof-intro} coincides with the standard ROF problem, the exponent of $h$ can be further increased to $\sfrac12 - \sfrac{1}{2q}$. Thus, in dimension $d=1$ there is a guaranteed rate of $h^{\sfrac14}$ for every $f \in L^2(\Omega)$. Moreover, the one-dimensional ROF model is known to be equivalent to an instance of the taut string problem as well as the one-dimensional TV flow, both in the continuous and discrete settings. For further details, see, for instance, \cite{Gra07,Mammen,SteWeiBroMraWel04} and \cite[Thm.~4.38]{Scherzer1}. Consequently, the estimates established here are also applicable to the solutions of these equivalent problems.

To the best of our knowledge convergence rates for $PCR$ approximations of problem \cref{eq:rof-intro} have not been published so far. In \cite{BarNocSal15} a rate of order $\mathcal{O}(h^{\sfrac12})$ was established for approximations of \cref{eq:ROF} by continuous piecewise (bi)linear functions on uniform triangular (rectangular) meshes assuming $f \in L^\infty(\Omega)$. Comparable error estimates of order $\mathcal{O}(h^\beta)$, $\beta \le \sfrac{1}{2}$, exist for various discretizations of the isotropic ROF model. See, for instance, \cite{Bar12, Bar21, CaiCha23, LaiLucWan09, WanLuc11}.

Returning to \cref{thm:approximationerror}, the specific error bound it establishes given $f \in L^q(\Omega)$, $2 \le q \le \infty$, reads
\begin{align}\label{eq:errorbound-intro}
	\norm*{ u_{A_Gf} - u_f}^{2}_{L^2} \le 2^{1+\frac{q'}{q}} J(u_f)^{1-\frac{q'}{q}} \norm*{u_f}_{L^q}^{\frac{q'}{q}} \left(\norm{u_f}_{L^q} + \norm{f}_{L^q} \right) h^{1-\frac{q'}{q}}.
\end{align}
Clearly, this estimate is only useful if $u_f \in L^q(\Omega)$ whenever $f \in L^q(\Omega)$. That this is indeed the case is a consequence of \cref{thm:invgeneral}, shown at the very end of this article. Its proof requires certain results concerning a discrete version of \cref{eq:rof-intro}, which we turn to next.

\paragraph{Taut strings and invariant $\varphi$-minimal sets} The taut string problem, which seems to originate in the operations research literature of the 1950s, see \cite{Dan71,ModHoh55}, can be stated as follows. Given two linear splines $F \le G$ agreeing at the endpoints of an interval $I$, find, among all linear splines $U$ satisfying $F \le U \le G$, the one with the shortest graph. Geometrically, the solution is found by means of the following eponymous procedure: Place a loose string inside the tube formed by $F$ and $G$ and pull taut. An illustration is given in \cref{fig:tautstring}.
\begin{figure}
	\centering
	\includegraphics[width=.6\textwidth]{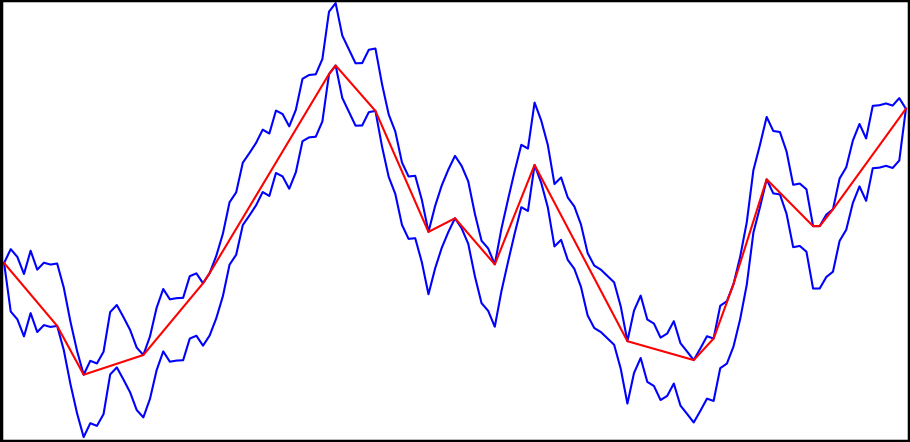}
	\caption{Example of a taut string (red) inside a tube.} \label{fig:tautstring}
\end{figure}
Besides having numerous applications ranging from statistics to communication theory and signal processing, taut strings enjoy a curious property, already pointed out in \cite{BBBB}. The taut string not only minimizes the graph length functional 
\begin{equation*}
	U \mapsto \int_I \sqrt{1+U'(x)^2} \, dx
\end{equation*}
among all linear splines lying between $F$ and $G$, but, in fact, it minimizes each of a whole class of problems where the function $y \mapsto \sqrt{1 + y^2}$ may be replaced by an arbitrary convex one.

Using an approach rooted in interpolation theory, this property of the taut string was further investigated in \cite{Kruglyak2}. There, a set $M \subset\mathbb{R}^n$ was termed invariant $\varphi$-minimal, if for every $a \in \mathbb{R}^n$ there is an $x_a \in M$ such that
\begin{equation}\label{eq:invphi}
	\sum^{n}_{i=1}\varphi\left(a_{i}-x_{a,i}\right)\leq\sum^{n}_{i=1}\varphi\left(a_{i}-x_{i}\right)
\end{equation}
for all $x \in M$ and all convex functions $\varphi: \mathbb{R} \to \mathbb{R}$. The set of derivatives of all admissible splines in the taut string problem is an example of such a set. A necessary and sufficient condition for a bounded, closed and convex set $M \subset \mathbb{R}^n$ to be invariant $\varphi$-minimal is that it is a convex polytope with all its faces being spanned by vectors of the form $e_i - e_j$, where $e_1,\ldots, e_n$ is the standard basis of $\mathbb{R}^n$. It turns out that these sets are precisely the base polyhedra associated to submodular functions on $2^{\{1,\ldots,n\}}$, see \cite[Thm.~17.1]{Fuj05} and \cite[Prop.\ 2.13]{KirSchSet19a}.

\paragraph{Universal minimality in the ROF model} Invariant $\varphi$-minimal sets also arise in connection with total variation minimization. An early indication of this circumstance was the result, established in \cite{Mammen}, that the discrete one-dimensional ROF problem is equivalent to an instance of the taut string problem. Then, in \cite[Thm.~4.46]{Scherzer1}, it was shown that the solution of the isotropic ROF model minimizes all functionals of the form
\begin{equation*}
	u \mapsto \int_\Omega \varphi(u(x))\, dx,
\end{equation*}
where $\varphi : \mathbb{R} \to \mathbb{R}$ is convex, over a neighbourhood of the datum $f \in L^2(\Omega)$. See also \cite[Lem.~1]{Ove19} for a similar result in dimension $d=1$. Another related result is \cite[Thm.\ 5.3, Rem.\ 5.4]{Kruglyak3} stating that the $L^1$-closure of the set of divergences of vector fields appearing in definition \cref{eq:J} is an invariant $\varphi$-minimal subset of $L^1(\Omega)$. Similar observations have been made in the graph setting. In \cite[Thm.\ 2.10]{KirSchSet19a} it was shown that the subdifferential of the graph total variation, given by $u \mapsto \sum_{v,\bar v}  \abs{u(v) - u(\bar v)}$, where $u$ is a real-valued vertex function and the sum runs over all pairs of adjacent vertices, is invariant $\varphi$-minimal. Consequently, the ROF minimizer on the graph, as considered there, has a universal minimality property analogous to that of the taut string.

\paragraph{The ROF model on weighted graphs} However, coming to back to problem \cref{eq:rof-intro} and its approximation by PCR functions, none of the results mentioned in the previous paragraph directly implies a similar universal minimality property for $u_f$ or its discrete approximation $u_{A_{G}f}$. We close this gap in the final sections of the present article.

The first step is \cref{prop:invphilambdaminimal}. It states that every bounded, closed and convex set $M\subset \mathbb{R}^n$ that is invariant $\varphi$-minimal also satisfies a weighted version of \cref{eq:invphi}. That is, for every $a\in\mathbb{R}^{n}$ and every $w\in\mathbb{R}^{n}_{>0}$ there is an $x_{a,w}\in M$ such that
\begin{align}\label{eq:phi-weight-intro}
	\sum^{n}_{i=1}w_{i}\varphi\left(\frac{a_{i}-x_{a,w,i}}{w_{i}}\right)\leq\sum^{n}_{i=1}w_{i}\varphi\left(\frac{a_{i}-x_{i}}{w_{i}}\right)
\end{align}
for all $x\in M$ and all convex $\varphi:\mathbb{R}\rightarrow\mathbb{R}$.

Now consider a graph $\mathsf{G} = (\sV,\sE)$ with vertex weights $\sw : \sV \to \mathbb{R}_{>0}$ and edge weights $\sW : \sE \to \mathbb{R}_{>0}$ as well as a given vertex function $\sff:\sV \to \mathbb{R}$ and the corresponding ROF problem on this graph

\begin{align}\label{eq:rof-graph-intro}
	\minimize_{\su:\sV \to \mathbb{R}} \quad  \frac{1}{2} \sum_{\sv \in \sV} \sw(\sv) \abs{\sff(\sv) - \su(\sv)}^2 + \alpha \sum_{(\sv,\bar \sv) \in \sE}  \sW(\sv, \bar \sv)  \abs{\su(\sv) - \su(\bar \sv)}. \tag{P$'$}
\end{align}
It follows from convex duality that \cref{eq:rof-graph-intro} is equivalent to
\begin{align*}
	\minimize_{\su \in \sff - \alpha \partial \sJ(0)} \quad  \sum_{\sv \in \sV} \sw(\sv) \abs{\su(\sv)}^2,
\end{align*}
where $\sJ$ is the weighted graph total variation, i.e.\ the second sum in \cref{eq:rof-graph-intro}. The above mentioned \cref{prop:invphilambdaminimal} combined with \cite[Thm.~2.4, Rem.~2.5]{Kruglyak3} implies that the solution $\su_{\sff}$ of \cref{eq:rof-graph-intro} minimizes not only the weighted $\ell^2$ norm over the set $\sff - \alpha \partial \sJ(0)$ but in fact every function of the form
\begin{align*}
	\su \mapsto \sum_{\sv \in \sV} \sw(\sv) \varphi(\su(\sv))
\end{align*}
where $\varphi: \mathbb{R} \to \mathbb{R}$ is convex. This is shown in \cref{thm:rofinvariancegraph} and generalizes \cite[Thm.\ 3.2]{KirSchSet19a}. 

\paragraph{Universal minimality in the anisotropic ROF} Crucially, the $PCR$ approximation $u_{A_G f}$ of $u_f$ can be identified with the solution of a problem of type \cref{eq:rof-graph-intro}. By means of an approximation argument we are therefore able to prove the analogue of \cref{thm:rofinvariancegraph} for problem \cref{eq:rof-intro}. Specifically, \cref{thm:invgeneral} states that, given datum $f\in L^{2}(\Omega)$, the inequality
\begin{align*}
	\int_{\Omega}\varphi(u_f(x))\, dx \le \int_{\Omega}\varphi(u(x)) \, dx
\end{align*}
holds for every $u \in f-\alpha\partial J(0)$ and every convex function $\varphi:\mathbb{R}\rightarrow\mathbb{R}$, where the set $f-\alpha\partial J(0)$ again arises in the convex dual of \cref{eq:rof-intro}. An immediate consequence, for example, is the estimate
\begin{align}\label{eq:main}
	\norm{u_f}_{L^q(\Omega)} \le \norm{f}_{L^q(\Omega)},
\end{align}
for $1 \le q \le \infty$. While \cref{thm:invgeneral} does not imply finiteness of these norms for $q>2$, we do have that $u_f$ belongs to $L^q(\Omega)$ whenever $f$ does. Note that this, in particular, validates the the error bound \cref{eq:errorbound-intro}. While \cref{thm:invgeneral} can be viewed as the analogue of \cite[Thm.~4.46]{Scherzer1} in $\ell^1$-anisotropy, it is noteworthy that the proofs of these two results are entirely different.

\paragraph{Outline} \Cref{sec:convex} collects results from convex analysis which apply to both \cref{eq:rof-intro} and \cref{eq:rof-graph-intro}. In \cref{sec:bv} we state an inequality of Poincar\'e type for functions of bounded variation which serves as the basis for later estimates. \Cref{sec:convergence} is devoted to the approximation space $PCR_G$ and properties of its associated $L^2$ projector $A_G$. The error estimate \cref{eq:errorbound-intro} is established in \cref{sec:setting}. \Cref{sec:phi} is about invariant $\varphi$-minimal sets in $\mathbb{R}^n$ and, in particular, the proof of \cref{eq:phi-weight-intro}. The ROF problem \cref{eq:rof-graph-intro} on the weighted graph is considered in \cref{sec:rof}. Finally, inequality \cref{eq:main} is established in \cref{sec:main}.
\section{Proximal mappings} \label{sec:convex}
In this section we collect some results about proximal mappings and support functions. For further details we refer to \cite{BauCom11}.

Below $\mathcal{H}$ denotes a real Hilbert space with inner product $\langle \cdot, \cdot \rangle$ and norm $\norm*{\cdot} = \langle \cdot, \cdot \rangle ^{\sfrac{1}{2}}$, and
\[F:\mathcal{H} \to \mathbb{R}\cup\{ + \infty\}\]
is a proper, convex and lower semicontinuous function. The \emph{subdifferential} of $F$ at $u \in \mathcal{H}$ is given by
\begin{equation*}
	\partial F (u) = \left\{w \in \mathcal{H} \mid F(v) - F(u) \ge \langle w, v - u\rangle \text{ for all } v \in \mathcal{H}\right\}.
\end{equation*}
If $F(u) \notin \mathbb{R}$, then $\partial F (u) = \emptyset$. For an arbitrary $f \in \mathcal{H}$ set
\begin{align*}
	E(u) \coloneqq \frac{1}{2} \norm*{ u - f }^2 + F(u), \quad u \in \mathcal{H}.
\end{align*}
\begin{proposition} \label[proposition]{thm:convexdiff}
For all $v \in \mathcal{H}$, all $u \in \mathcal{H}$ such that $\partial E(u) \neq \emptyset$ and all $w \in \partial E(u)$
\begin{equation}
	E(v) - E(u) \geq \frac{1}{2} \norm*{v - u}^2 + \langle w,v-u\rangle.
\end{equation}
\end{proposition}
\begin{proof}
Note that
\begin{equation}\label{eq:l2difference}
%\begin{aligned}
	\norm*{ v - f }^2 - \norm*{u - f }^2 = \norm*{v-u}^2 + 2 \langle u-f,v-u \rangle.
%\end{aligned}
\end{equation}
Now, for all $y\in \partial F(u)$ we have
\begin{align*}
	E(v) - E(u)
		&=		\frac12 \norm*{ v - f}^2 - \frac12 \norm*{ u - f}^2 + F(v) - F(u) \\
		&\geq	\frac{1}{2} \norm*{v-u}^2 + \langle u - f,v-u \rangle + \langle y,v-u \rangle \\
		&=		\frac{1}{2} \norm*{v-u}^2 + \langle u - f + y,v-u\rangle.
\end{align*}
This ends the proof, because $\left\{u-f + y \mid y\in\partial F(u)\right\}=\partial E(u)$.
\end{proof}
Every proper, convex and lower semicontinuous function $F$ has a well-defined \emph{proximal mapping}, $\operatorname{prox}_F: \mathcal{H} \to \mathcal{H}$, which is given by
\[ \operatorname{prox}_F (f) = \argmin_{u \in \mathcal{H}} \left\{ \frac{1}{2} \norm*{u - f}^2 + F(u) \right\}.\]
\begin{proposition}\label[proposition]{thm:prox-nonexpansive}
For all $f,g \in \mathcal{H}$
\[ \norm*{\operatorname{prox}_F (f) - \operatorname{prox}_F (g)} \le \norm*{f-g}. \]
\end{proposition}
\begin{proof}
See \cite[Prop.~12.27]{BauCom11}.
\end{proof}
Next, we turn to a particular class of proper, convex and lower semicontinuous functions on $\mathcal{H}$. The \emph{support function} of a set $S\subset \mathcal{H}$ is defined as
\begin{align}\label{eq:support-fct}
	\sigma_S(u) = \sup_{v \in S} \langle v, u \rangle, \quad u \in \mathcal{H}.
\end{align}
The convex hull of $S$ is denoted by $\operatorname{co} S$ and its closure in $\mathcal{H}$ by $\overline{S}$.
\begin{proposition}\label[proposition]{thm:support-fct}
Let $S \subset \mathcal{H}$ be nonempty. Then $\sigma_S = \sigma_{\overline{\operatorname{co}S}}$ is proper, convex and lower semicontinuous and $\partial \sigma_S (0) = \overline{\operatorname{co}S}$.
\end{proposition}
\begin{proof}
See \cite[Props.~7.11, 14.11, 16.18]{BauCom11}.
\end{proof}
\begin{proposition}\label[proposition]{thm:dual2}
For all nonempty $S\subset \mathcal{H}$ and all $f \in \mathcal{H}$
\[ \operatorname{prox}_{\sigma_S} (f) = \argmin_{u \in f-\overline{\operatorname{co}S}} \norm*{u}.\]
\end{proposition}
\begin{proof}
This is a consequence of convex duality. The convex conjugate of $\sigma_S$ is the characteristic function of $\overline{\operatorname{co}S}$, that is, $\sigma_S^*(w) = 0$ for $w \in \overline{\operatorname{co}S}$ and $\sigma_S^*(w) = +\infty$ otherwise. Therefore, the dual problem of 
\begin{align*}
	\minimize_{u \in \mathcal{H}} \quad \frac12 \norm*{u-f}^2 + \sigma_S(u) 
\end{align*}
is equivalent to
\begin{align*}
	\minimize_{w \in \overline{\operatorname{co}S}} \quad \norm*{w-f}
\end{align*}
and according to \cite[Prop.~15.13, Thm.~19.1]{BauCom11} the primal solution $\hat u$ is related to the dual solution $\hat w$ by $\hat u = f - \hat w$. Hence
\begin{align*}
	\norm*{\operatorname{prox}_{\sigma_S} (f)} = \norm*{\hat u} = \norm*{f - \hat w} = \min_{w \in \overline{\operatorname{co}S}} \norm*{f-w} = \min_{v \in f - \overline{\operatorname{co}S}} \norm*{v}.
\end{align*}
\end{proof}
\begin{proposition}\label[proposition]{thm:diminishing}
Let $F$ be the support function of a nonempty subset of $\mathcal{H}$, and let $A:\mathcal{H} \to \mathcal{H}$ be selfadjoint. Then the following two statements are equivalent.
\begin{enumerate}
	\item $A(\partial F(0)) \subset \partial F(0)$
	\item $F(Au) \le F(u)$  for all $u \in \mathcal{H}$
\end{enumerate}
\end{proposition}
\begin{proof}
Due to \cref{thm:support-fct} $F$ is the support function of $\partial F(0)$. The selfadjointness of $A$ now implies
        \begin{align*}
            F(Au)   & % =   \sup_{v \in \partial F(0)} \langle Au,v \rangle
                    =   \sup_{v \in \partial F(0)} \langle u,Av \rangle
                    =   \sup_{w \in A(\partial F(0))} \langle u,w \rangle %\\
                    \le \sup_{w \in \partial F(0)} \langle u,w \rangle
                    = F(u).
        \end{align*}
Conversely, let $v = Aw$ for some $w \in \partial F(0)$. Then
        \begin{align*}
            \langle v,u \rangle = \langle Aw , u \rangle = \langle w , Au \rangle \le F(Au) \le F(u)
        \end{align*}
        for all $u \in \mathcal{H}$. Hence $v \in \partial F(0).$
\end{proof}
\section{Total variation in $\ell^1$-anisotropy} \label{sec:bv}
From now on $\Omega \subset \mathbb{R}^d$, $d\ge 1$, denotes an open hyperrectangle,
\begin{align}\label{eq:Omega}
	\Omega = a + \prod_{i=1}^d \left( 0,l_i \right),
\end{align}
where $a \in \mathbb{R}^d$ and $l \in \mathbb{R}^d_{>0}$. The $d$-dimensional volume of $\Omega$ is denoted by $\abs*{\Omega}$ and the average of $g \in L^1(\Omega)$ by
\begin{equation}\label{eq:average}
	g_{\Omega} \coloneqq \frac{1}{\abs*{\Omega}}\int_{\Omega}g \, dx.
\end{equation}
Define
\begin{align}\label{eq:B}
	\mathcal{B}(\Omega) = \left\{ H \in C_c^\infty \left(\Omega, \mathbb{R}^d\right) \,\middle|\, \abs*{H(x)}_\infty \le 1 \text{ for all } x \in \Omega \right\}.
\end{align}
The $\ell^1$-anisotropic total variation on $\Omega$ of $g \in L^1(\Omega)$ is given by
\begin{align}\label{eq:TV}
	\operatorname{TV}_{1,\Omega}(g) = \sup_{H \in \mathcal{B}(\Omega)}  \int_\Omega g(x) \divergence H (x) \, dx.
\end{align}
The space of functions of bounded variation on $\Omega$,
\begin{align*}
	BV(\Omega) = \left\{ u \in L^1(\Omega) \,\middle|\, \operatorname{TV}_{1,\Omega}(u) < + \infty \right\},
\end{align*}
is unaffected by this choice of anisotropy.

We adopt the conventions
\begin{equation*}
	1^* \coloneqq
	\begin{dcases}
		\infty, & d=1,\\
		\frac{d}{d-1}, & d\geq 2,
	\end{dcases}
\end{equation*}
and $\sfrac{1}{p} \coloneqq 0$ when $p=\infty$. 

The following inequality of Poincar\'{e} type serves as the basis for subsequent approximations.
\begin{lemma} \label[lemma]{thm:poincare} Let $d \ge 1$ and $1\leq p\leq 1^*$. There is a constant $C = C(d,p) > 0$ such that
\begin{equation*}
	\norm*{ u - u_{\Omega}}_{L^p(\Omega)} \leq C \frac{\abs*{l}_\infty}{\abs*{\Omega}^{1-\frac{1}{p}}} \operatorname{TV}_{1,\Omega}(u),
\end{equation*}
for all $u \in BV(\Omega)$.
\end{lemma}
\begin{proof}
	Let $I = (0,1)$ and let $C = C(d,p) > 0$ be such that
	\begin{equation} \label{eq:gns}
		\norm*{\hat{u}-\hat{u}_{I^d}}_{L^{p}(I^d)}\leq C \operatorname{TV}_{1,I^d}(\hat{u})
	\end{equation}
	holds for all $\hat{u}\in BV(I^d)$, see e.g.\ \cite[Rem.~3.50]{AmbFusPal00}. By virtue of the transformation $ \phi : I^d \to \Omega$, $\phi(z) = \operatorname{diag}(l)z + a$, every $u\in BV(\Omega)$ can be identified with $\hat{u} = u \circ \phi \in BV(I^d)$. Taking into account \cref{eq:gns} and the fact that $\hat{u}_{I^d} = u_{\Omega}$, we obtain
	\begin{equation} \label{eq:estimaterectangle}
	\begin{aligned}
		\norm*{u-u_\Omega}_{L^p(\Omega)} = \abs*{\Omega}^{\frac{1}{p}} \norm*{\hat{u} - \hat{u}_{I^d}}_{L^p(I^d)} \leq C_d \abs*{\Omega}^{\frac{1}{p}}\operatorname{TV}_{1,I^d}(\hat{u}).
	\end{aligned}
	\end{equation}
	Note further that
	\begin{equation*}
	\begin{aligned}
		\operatorname{TV}_{1,I^d}(\hat{u})
			&=\sup_{\hat{H} \in \mathcal{B}(I^d)} \int_{I^d} \hat{u} \divergence \hat{H} dz
			=\sup_{H\in\mathcal{B}(\Omega)}\int_{\Omega} u \sum_{j=1}^d l_j \frac{\partial H}{\partial y_j} \frac{1}{\abs*{\Omega}}dy \\
			&\leq\frac{\abs*{l}_\infty}{\abs*{\Omega}}\sup_{H\in\mathcal{B}(\Omega)}\int_{\Omega}u \divergence H\,dy
			=\frac{\abs*{l}_\infty}{\abs*{\Omega}}\operatorname{TV}_{1,\Omega}(u).
	\end{aligned}
	\end{equation*}
	Applying the preceding estimate in \cref{eq:estimaterectangle} concludes the proof.
\end{proof}
For our purposes \cref{thm:poincare} is most useful when $p=1$ or $d=1$. See \cref{rem:best-exponent} below for further details. In order to simplify subsequent error estimates, while also making them dimension independent, we briefly point out next that $C=1$ is a possible choice in these two cases.

We recall the interpolation inequality for $L^p$ norms. Let $1\leq p\le r \le s\leq\infty$ and $g\in L^p\cap L^s$. Then $g\in L^r$ and
\begin{equation} \label{eq:interpolationineq}
\norm{g}_{L^r} \leq \norm{g}_{L^p}^{t} \norm{g}_{L^s}^{1-t},
\end{equation}
where $\frac{1}{r} = \frac{t}{p} + \frac{1-t}{s}$.

\begin{corollary} \label[corollary]{thm:constant1}
	Let $d \ge 1$. Then
	\begin{align*}
		\norm{u - u_{\Omega}}_{L^1(\Omega)} \le \abs{l}_\infty \operatorname{TV}_{1,\Omega}(u)
	\end{align*}
		for all $u \in BV(\Omega).$ If $d=1$, then
	\begin{align*}
		\norm{u - u_{\Omega}}_{L^p(\Omega)} \le \abs{\Omega}^\frac{1}{p} \operatorname{TV}_{1,\Omega}(u)
	\end{align*}
	for all $u \in BV(\Omega)$ and $1 \le p \le \infty$.
\end{corollary}

\begin{proof}
We only need to show that $C=1$ is possible in \cref{eq:gns} if $p=1$ or $d=1$.

Concerning $p=1$ this can be done with a slight modification of the proof of \cite[Prop.\ 12.29]{Leo09}, which concerns the isotropic Sobolev seminorms $(\int_\Omega \abs{\nabla u(x)}_2^p \, dx)^{\sfrac{1}{p}}$.

Regarding $d=1$, we have $\abs{u(x) - u_{I}} \le \int_0^1 |u'(y)|\,dy$ for every $u \in C^\infty(\bar I)$ and $x \in I$. A density argument shows that $\norm{u - u_{I}}_{L^\infty(I)} \le \operatorname{TV}_{1,I}(u)$ for all $u \in BV(I)$. Interpolating between $L^1$ and $L^\infty$ according to \cref{eq:interpolationineq} proves the claim.
\end{proof}

\section{Averaging on rectilinear grids}\label{sec:convergence}
Let $G = G(\Omega)$ be a \emph{grid} on $\Omega$, that is, a finite collection of affine hyperplanes of $\mathbb{R}^d$, each orthogonal to one of the coordinate axes. In addition, we assume that $G$ covers the entire boundary of $\Omega$ (recall \cref{eq:Omega}), that is, $\partial \Omega \subset \bigcup G$. The grid $G$ defines a natural \emph{partition}
\begin{equation*}
\mathcal{Q}(G) = \{R_1,\ldots, R_n\}
\end{equation*} of $\Omega$ into smaller open hyperrectangles. Note that in the particular case $d=1$, $G$ is a finite set of real numbers which induces a partition of $\Omega=(a,a+l)$ into open subintervals.

We denote by $PCR_G(\Omega)$, or simply $PCR_G$, the set of all functions $g:\Omega \to \mathbb{R}$ which are almost everywhere equal to finite linear combinations of indicator functions of elements of $\mathcal{Q}(G).$ That is, a $g \in PCR_G$ satisfies for almost every $x \in \Omega$
\[ g(x) = \sum_{i=1}^n g_i \mathbf{1}_{R_i}(x),\]
where $g_i = g_{R_i} \in \mathbb{R}$ and $\mathbf{1}_{S}(x) = 1$  if $x\in S$ and $\mathbf{1}_{S}(x)=0$ if $x \notin S$. For every $g \in PCR_G$
	\begin{align}\label{eq:pcr-tv}
	\operatorname{TV}_{1,\Omega}(g) = \sum_{i,j=1}^n H^{d-1}\big(\overline{R_i} \cap \overline{R_j}\big) \abs{g_i-g_j} < \infty,
	\end{align}
where $H^{d-1}$ is the Haus\-dorff measure of dimension $d-1$.

The \emph{averaging operator} $A_{G}:L^{1}(\Omega)\rightarrow PCR_{G}\left(\Omega\right)$ associated to $G$ is defined by
\begin{align*}
A_{G}g=\sum^{n}_{i=1}g_{R_i}\mathbf{1}_{R_{i}},
\end{align*}
recall \cref{eq:average}. The averaging operator satisfies
\begin{align}\label{eq:selfadjoint}
\int_\Omega g(x) A_G u(x) \, dx = \int_\Omega u(x) A_G g(x)\,dx 
\end{align}
for all $g,u \in L^1(\Omega).$ It follows that, on $L^2(\Omega)$, $A_G$ is the orthogonal projector onto $PCR_G$. Moreover, it is clear that
\begin{align}\label{eq:nonexpansive}
\|A_G g\|_{L^p(\Omega)} \le \|g\|_{L^p(\Omega)}
\end{align}
holds for $p = \infty$. That this inequality is valid also for $1\le p < \infty$ is a consequence of the next lemma.
\begin{lemma} \label[lemma]{lemma:operatorA}
	For every $g\in L^{1}(\Omega)$ and convex $\varphi:\mathbb{R} \to \mathbb{R}$
	\begin{align*}
	\int_{\Omega}\varphi\left(A_{G}g(x)\right) dx \leq \int_{\Omega}\varphi\left(g(x)\right) dx.
	\end{align*}
\end{lemma}
\begin{proof}
	See \cite[Lem.\ 2]{KirSchSet19}.
\end{proof}

Let $D(G)$ denote the maximal diameter among all hyperrectangles in the partition $\mathcal{Q}(G)$ of $\Omega$. Averaging on finer and finer partitions of $\Omega$ results in a convergent procedure:
\begin{lemma}\label[lemma]{thm:convergenceaverage}
	Consider a sequence of grids $(G_k)$ such that $D(G_k)\to 0$ as $k\to\infty$. Then
	\begin{equation*}
	\norm{g-A_{G_k}g}_{L^p(\Omega)}\to 0
	\end{equation*}
	for $g\in L^{p}(\Omega)$, $1\leq p<\infty$.
\end{lemma}
\begin{proof}
	Let $g\in L^p(\Omega)$ and $\epsilon>0$. Take an $\eta \in C_c^\infty(\Omega)$ such that $\norm{g-\eta}_{L^p(\Omega)}<\epsilon$. As $\eta$ is uniformly continuous on $\Omega$, there is a $\delta>0$ for which
	\begin{equation} \label{eq:uniform}
	\abs{\eta(x)-\eta(y)}<\epsilon
	\end{equation}
	holds for every $x,y\in\Omega$ with $\abs{x-y}<\delta$. Let $D(G_k)\leq\delta$ and consider a hyperrectangle $R\in\mathcal{Q}(G_k)$. Applying the intermediate value theorem, there exists $\xi \in R$ satisfying $\eta(\xi) = \eta_R$.
	Taking into account \cref{eq:uniform}, we have for any $x\in R$
	\begin{equation} \label{eq:uniform_pcr}
	\abs{\eta(x)-A_{G_k}\eta(x)} = \abs{\eta(x)-\eta(\xi)}<\epsilon.
	\end{equation}
	As $R\in\mathcal{Q}(G_k)$ was chosen arbitrarily, it follows from \cref{eq:uniform_pcr} that
	\begin{equation*}
	\abs{\eta(x)-A_{G_k}\eta(x)} < \epsilon,\text{ a.e. } x\in\Omega,
	\end{equation*}
	and therefore
	\begin{equation*}
	\norm{\eta-A_{G_k}\eta}_{L^p(\Omega)} < \epsilon \abs{\Omega}^{\frac{1}{p}}.
	\end{equation*}
	Using the nonexpansiveness of $A_{G_k}$ on $L^p(\Omega)$, we derive the estimate
	\begin{align*}
	\norm{g-A_{G_k}g}_{L^p(\Omega)} \leq 2 \norm{g-\eta}_{L^p(\Omega)} + \norm{\eta-A_{G_k}\eta}_{L^p(\Omega)} < 2\epsilon+\epsilon \abs{\Omega}^{\frac{1}{p}}
	\end{align*}
	and the lemma is proved.
\end{proof}

We next turn to approximation of $u \in BV(\Omega)$ by $A_Gu$. See also \cite[Thm.~3.9]{CasKunPol99}. The vector of side lengths of $R_i \in \mathcal{Q}(G)$ is denoted by $r_i = (r_{i,1},\ldots, r_{i,d}) \in \mathbb{R}^d_{>0}$. We also introduce the notation
\begin{align*}
	h		&= h(G) = \max_{1\leq i\leq n}\abs*{r_i}_\infty
\end{align*}
for the greatest side length among all hyperrectangles in the partition $\mathcal{Q}(G)$.
\begin{theorem} \label{thm:rateconvaveraging}
	Let $d\ge 1$ and $1\leq p\leq 1^*$. There is a constant $C=C(d,p)>0$ such that
	\begin{equation}\label{eq:rateconvaveraging}
		\norm*{u-A_Gu}_{L^p(\Omega)} \leq C \max_{1 \le i \le n} \frac{\abs*{r_i}_\infty}{\abs*{R_i}^{1-\frac{1}{p}}} \operatorname{TV}_{1,\Omega}(u)
	\end{equation}
	for all $u \in BV(\Omega)$.
\end{theorem}
\begin{proof} 
	Taking into account \cref{thm:poincare}, we derive
	\begin{equation*}
	\begin{aligned}
		\norm*{u- A_{G}u}_{L^p(\Omega)}
			&\leq\sum_{j=1}^{n} \norm*{u-u_{R_j}}_{L^p(R_j)}
				\leq C \sum_{j=1}^{n} \frac{\abs*{r_{j}}_\infty}{\abs*{R_j}^{1-\frac{1}{p}}}\operatorname{TV}_{1,R_j}(u) \\
			&\leq C \max_{1 \le i \le n} \frac{\abs*{r_i}_\infty}{\abs*{R_i}^{1-\frac{1}{p}}}\sum_{j=1}^{n}\operatorname{TV}_{1,R_j}(u)
				\leq C \max_{1 \le i \le n} \frac{\abs*{r_i}_\infty}{\abs*{R_i}^{1-\frac{1}{p}}}\operatorname{TV}_{1,\Omega}(u).
	\end{aligned}
	\end{equation*}
\end{proof}
At this point we recall \cref{thm:constant1} and the two corresponding special cases of \cref{eq:rateconvaveraging} which will be used below: First, for $p=1$ we obtain
\begin{align}\label{eq:averaging-error-p1}
	\norm*{u-A_Gu}_{L^1(\Omega)} \leq h \operatorname{TV}_{1,\Omega}(u).
\end{align}
Second, in dimension $d=1$, where $\Omega$ is a bounded interval, we have
\begin{align}\label{eq:averaging-error-d1}
	\norm*{u-A_Gu}_{L^p(\Omega)} \leq h^{\frac{1}{p}} \operatorname{TV}_{1,\Omega}(u)
\end{align}
for $1 \le p \le \infty$.

The restriction of $\operatorname{TV}_{1,\Omega}$ to the Hilbert space $L^2(\Omega)$ is denoted by $J$, that is,
\begin{align*}
J : L^2(\Omega) \to [0,+\infty], \quad J(u) = \operatorname{TV}_{1,\Omega}(u).
\end{align*}
Note that $J = \sigma_{\divergence(\mathcal{B}(\Omega))}$. Therefore, due to \cref{thm:support-fct}, $\partial J(0) = \overline{\divergence ( \mathcal{B}(\Omega))}$, where the bar denotes closure in $L^2(\Omega)$.

The following result, stating that $A_G$ maps $\partial J(0)$ into itself, crucially relies on the fact that the rectilinearity of the grid is compatible with the anisotropy of $J$. 
\begin{theorem} \label{prop:projection}
	$A_{G}\left(\partial J(0)\right) \subset \partial J(0)$.
\end{theorem}
\begin{proof}
	See \cite[Thm.\ 1]{KirSchSet19}. Note that the additional assumption in \cite[Thm.\ 1]{KirSchSet19}, that $G$ should cover $\partial \Omega$, was already imposed at the beginning of \cref{sec:convergence}.
\end{proof}
\Cref{prop:projection} and \cref{thm:diminishing} imply that $A_G$ is total variation diminishing.
\begin{corollary} \label[corollary]{thm:tvd} $J(A_Gu) \le J(u)$ for all $u \in L^2(\Omega)$.
\end{corollary}
\section{Error bounds for the $\ell^1$-anisotropic ROF model} \label{sec:setting}
Let $\Omega\subset \mathbb{R}^d$ be as in \cref{eq:Omega}. Given $f \in L^2(\Omega)$ and $\alpha >0$ we consider the problem
\begin{align}\label{eq:ROF}
	\minimize_{u \in L^2(\Omega)} \quad E(u;f),  \tag{P}
\end{align}
where
\begin{align*}
	 E(u;f) = \frac12 \norm*{ u-f }^2_{L^2(\Omega)} + \alpha J(u).
\end{align*}
The results of \cref{sec:convex} imply the well-posedness of problem \cref{eq:ROF}: It is uniquely solvable and stable with respect to perturbations in $f$. We set
\begin{align*}
	u_f = \argmin_{u \in L^2(\Omega)} E(u;f).
\end{align*}

\Cref{thm:fpcrupcr} below asserts that the $\ell^{1}$-anisotropic ROF model respects the piecewise constant structure of $PCR_{G}$ data. A proof for $d=1$, where the situation is comparatively simple, can be found in \cite[Lem.\ 4.34]{Scherzer1}. The case $d=2$ was resolved in \cite{Lasica1} while the general result ($d\ge 1$) was established in \cite{KirSchSet19}.
\begin{theorem}\label{thm:fpcrupcr}
If $f\in PCR_{G}$, then $u_f \in PCR_{G}$.
\end{theorem}
\begin{proof}
The proof combines \cref{thm:dual2} and \cref{lemma:operatorA} with \cref{prop:projection}. For details we refer to \cite[Thm.\ 2]{KirSchSet19} from which \cref{thm:fpcrupcr} follows.
\end{proof}

\begin{remark}\label[remark]{rem:uAf} %\color{blue}
\Cref{thm:fpcrupcr} implies that, for $PCR_{G}$ datum, \cref{eq:ROF} is a finite-dimen\-sion\-al problem. In fact, if $u,f \in PCR_{G}$, then
\begin{align}\label{eq:discrete-energy}
	E(u;f) = \frac12 \sum_{i=1}^n \abs{R_i} \abs{u_i-f_i}^2 + \alpha \sum_{i,j=1}^n H^{d-1}\big(\overline{R_i} \cap \overline{R_j}\big) \abs{u_i-u_j},
\end{align}
recall \cref{eq:pcr-tv}. The discrete energies arising in this way are considered in more detail in the following sections. For now we only add that
\begin{equation}\label{eq:ortho}
	E(u;f) = E(u;A_Gf) + \frac12 \norm{A_G f-f}_{L^2(\Omega)}^2
\end{equation}
for all $f\in L^2(\Omega)$ and $u \in PCR_G$ and therefore %the following three minimization problems are equivalent.
\begin{align}\label{eq:equivalencemin}
	u_{A_G f} = \argmin_{u \in PCR_G} E(u;A_G f) = \argmin_{u \in PCR_G} E(u;f).
\end{align}
In the remainder of this section we investigate how well $u_{A_G f}$ approximates $u_f$.
\end{remark}
\begin{theorem}[Error estimates]\label{thm:approximationerror}$ $
	\begin{enumerate}
		\item For every $f \in L^2(\Omega)$
			\begin{align} \label{eq:estimate1}
				\norm*{u_{A_Gf} - u_f }_{L^2} \le \norm*{ A_Gf - f }_{L^2}.
			\end{align}
		\item For every $f \in L^q(\Omega)$, $2 \le q \le \infty$,
			\begin{align} \label{eq:estimate2}
				\norm*{ u_{A_Gf} - u_f}^{2}_{L^2} \le 2^{1+\frac{q'}{q}} J(u_f)^{1-\frac{q'}{q}} \norm*{u_f}_{L^q}^{\frac{q'}{q}} \left(\norm{u_f}_{L^q} + \norm{f}_{L^q} \right) h^{1-\frac{q'}{q}}.
			\end{align}
			If in addition $d=1$, then
			\begin{align} \label{eq:estimate1d}
				\norm*{ u_{A_Gf} - u_f}^{2}_{L^2} \le 2 J(u_f) (\norm{u_f}_{L^q}+\norm{f}_{L^q}) h^{1-\frac{1}{q}}.
			\end{align}
	\end{enumerate}
\end{theorem}
\begin{proof}
	Estimate \cref{eq:estimate1} follows from \cref{thm:prox-nonexpansive}.
	The proof of \cref{eq:estimate2} is similar to \cite[Thm.~6.5]{BarNocSal15} but differs in certain key points. Recalling \cref{thm:convexdiff} (noting that $0\in\partial E(u_f)$), the equivalence \cref{eq:equivalencemin}, the total variation diminishing property (\cref{thm:tvd}) and the nonexpansiveness of $A_G$ (inequality \cref{eq:nonexpansive}) we obtain
	\begin{align}
		\norm{ u_{A_Gf}-u_f}_{L^2}^2
			&\leq	2(E(u_{A_Gf};f)-E(u_f;f)) \leq 2(E(A_Gu_f;f)-E(u_f;f)) \nonumber \\
			%&=\lVert u_{A_Gf}-f\rVert_{L^2}^2+\alpha J(u_{A_Gf})-\lVert u_f -f\rVert_{L^2}^2-\alpha J(u_f) \nonumber \\
			&=		\norm{ A_{G}u_f-f}_{L^2}^2 + 2\alpha J(A_Gu_f)-\norm{ u_f -f}_{L^2}^2 - 2\alpha J(u_f) \nonumber \\
			&\leq	\norm{ A_{G}u_f-f}_{L^2}^2-\norm{ u_f -f}_{L^2}^2 = \langle A_{G}u_f-u_f,A_{G}u_f+u_f-2f\rangle \nonumber \\
			&\leq	\norm{ A_{G}u_f-u_f}_{L^{q^{\prime}}}\norm{ A_{G}u_f+u_f-2f}_{L^q} \nonumber \\
			&\leq	2\norm{ A_{G}u_f-u_f}_{L^{q^{\prime}}}(\norm{ u_f}_{L^q}+\norm{ f}_{L^q}). \label{eq:chain}
		\intertext{Now, the interpolation inequality \cref{eq:interpolationineq} as well as \cref{eq:averaging-error-p1} imply}
			&\leq	2\norm{ A_{G}u_f-u_f}_{L^1}^{1-\frac{q'}{q}}\norm{ A_{G}u_f-u_f}_{L^{q}}^{\frac{q'}{q}}(\norm{ u_f}_{L^q}+\norm{ f}_{L^q}) \nonumber \\
			&\leq	2^{1+\frac{q'}{q}}\norm{ A_{G}u_f-u_f}_{L^1}^{1-\frac{q'}{q}}\norm{u_f}_{L^q}^{\frac{q'}{q}}(\norm{u_f}_{L^q}+\norm{ f}_{L^q}) \nonumber \\
			&\leq 2^{1+\frac{q'}{q}} h^{1-\frac{q'}{q}} J(u_f)^{1-\frac{q'}{q}}\norm{u_f}_{L^q}^{\frac{q'}{q}}(\norm{u_f}_{L^q}+\norm{f}_{L^q})  \nonumber 
	\end{align}
	and we have proved \cref{eq:estimate2}.
	
	In dimension $d=1$, we directly estimate \cref{eq:chain} using \cref{eq:averaging-error-d1}. This shows \cref{eq:estimate1d}.
\end{proof}
\begin{remark} \label[remark]{rem:finite-norm}
We stress that the norm $\norm*{ u_f}_{L^q}$ appearing on the right hand side of \cref{eq:estimate2} is bounded by $\norm*{f}_{L^q}$ and therefore finite. See \cref{thm:invgeneral} and \cref{rem:prox-nonexpansive} below.
\end{remark}
\begin{remark}\label[remark]{rem:best-exponent}
In the proof of \cref{eq:estimate2}, first, Hölder's inequality was applied with exponents $q$, $q'$ and then the interpolation inequality \cref{eq:interpolationineq} with $p=1$, $r=q'$, $s = q$. It can be shown that these choices maximize the resulting exponent of $h$ provided that $d\geq 2$. That is, for $d\geq 2$ we cannot increase the exponent beyond $1-q'/q$ within the framework of the present proof.
\end{remark}
The following corollary supplements the convergence rates for the $L^2$ approximation error, which are immediate from \cref{thm:approximationerror}, with rates for the total variations $J(u_{A_{G}f})$ and the discrete energies $E(u_{A_{G}f};A_{G}f)$, recall \cref{eq:discrete-energy}.

\begin{corollary}[Convergence rates] \label[corollary]{thm:rate}
Let $(G_k)$ be a sequence of grids such that $h(G_k) \to 0$ as $k \to \infty$. Then the following statements hold.
\begin{enumerate}
	\item If $f \in L^2(\Omega)$, then
		\begin{subequations} \label{eq:rate1}
		\begin{gather}
			\norm[\big]{ u_{A_{G_k}f} - u_f}_{L^2} \to 0, \label{eq:rate1-l2} \\
			J(u_{A_{G_k}f}) \to J(u_f), \label{eq:rate1-J} \\
			E(u_{A_{G_k}f};A_{G_k}f) \to E(u_f;f). \label{eq:rate1-E} 
		\end{gather}
		\end{subequations}
	\item \label{it:rate} If $f \in L^q(\Omega)$, $2 \le q \le \infty$, then
		\begin{subequations}\label{eq:rate2}
			\begin{gather}
				\norm[\big]{ u_{A_{G_k}f} - u_f}_{L^2} = \mathcal{O}\big(h(G_k)^{\frac12 - \frac{q'}{2q}}\big), \label{eq:rate2-l2} \\
				\abs[\big]{ J(u_{A_{G_k}f}) - J(u_f)} = \mathcal{O}\big(h(G_k)^{\frac12 - \frac{q'}{2q}}\big). \label{eq:rate2-J}
			\end{gather}
		\end{subequations}
		If in addition $d=1$, then
		\begin{subequations}\label{eq:rate1d}
			\begin{gather}
				\norm[\big]{ u_{A_{G_k}f} - u_f}_{L^2} = \mathcal{O}\big(h(G_k)^{\frac12 - \frac{1}{2q}}\big), \label{eq:rate1d-l2} \\
				\abs[\big]{ J(u_{A_{G_k}f}) - J(u_f)} = \mathcal{O}\big(h(G_k)^{\frac12 - \frac{1}{2q}}\big). \label{eq:rate1d-J}
			\end{gather}
		\end{subequations}
	\item If $f \in L^q(\Omega) \cap BV(\Omega)$, $2 \le q \le \infty$, then
		\begin{align} \label{eq:rate3}
			\abs{E(u_{A_{G_k}f};A_{G_k}f) - E(u_f;f)} = \mathcal{O}\big(h(G_k)^{1 - \frac{q'}{q}}\big).
		\end{align}
		If in addition $d=1$, then
		\begin{align} \label{eq:rate3-1d}
			\abs{E(u_{A_{G_k}f};A_{G_k}f) - E(u_f;f)} = \mathcal{O}\big(h(G_k)^{1 - \frac{1}{q}}\big).
		\end{align}
	%\item If $f \in L^q(\Omega) \cap BV(\Omega)$, $2 \le q \le \infty,$
\end{enumerate}
\end{corollary}
\begin{proof}
The first claim, \cref{eq:rate1-l2}, follows from \cref{eq:estimate1}, because $h(G_k)\to 0$ implies $D(G_k)\to 0$ and therefore $A_{G_k}f \to f$ in $L^2(\Omega)$ as \cref{thm:convergenceaverage} shows. Next, note that
\begin{equation} \label{eq:E}
	E(u_f;f)\leq E(u_{A_{G_k}f};f)\leq E(A_{G_k}u_f;f) \to E(u_f;f),
\end{equation}
where the second inequality is due to \cref{rem:uAf}, the convergence $\norm{A_{G_k}u_f - f}_{L^2}^2 \to \norm{u_f - f}_{L^2}^2$ follows from \cref{thm:convergenceaverage}, and $J(A_{G_k}u_f) \to J(u_f)$ holds because of the lower semicontinuity of $J$ and the total variation diminishing property of $A_{G_k}$. Since we also have $\norm{u_{A_{G_k}f}-f}_{L^2}^2 \to \norm{u_{f}-f}_{L^2}^2$, \cref{eq:E} implies \cref{eq:rate1-J} and in further consequence \cref{eq:rate1-E}.

The rates \cref{eq:rate2-l2} and \cref{eq:rate1d-l2} are immediate consequences of the error estimates \cref{eq:estimate2} and \cref{eq:estimate1d}, respectively, taking into account \cref{rem:finite-norm}. Regarding \cref{eq:rate2-J} we estimate the difference of data terms using \cref{eq:l2difference} and the Cauchy-Schwarz inequality
\begin{align*} %\label{2}
    \abs*{\frac12 \norm{u_{A_{G_k}f} - f}_{L^2}^2 - \frac12 \norm{u_f - f}_{L^2}^2}
    	&\le \frac12 \norm{u_{A_{G_k}f} - u_f}^2_{L^2} + \abs{\langle u_{A_{G_k}f} - u_f,u_f-f\rangle} \\
    	&\le \frac{B}{2} + \norm{u_f-f}_{L^2} \sqrt{B},
\end{align*}
where $B = B(h(G_k))$ stands for the right-hand side in \cref{eq:estimate2}. Now the inequality
\begin{align}\label{eq:energy-difference}
	0 \le E(u_{A_{G_k}f};f) - E(u_f;f) \le \frac{B}{2},
\end{align}
which is a consequence of the proof of \cref{thm:approximationerror}, implies
\begin{align*}
	-\frac{B}{2} - \norm{u_f-f}_{L^2} \sqrt{B}\le \alpha J(u_{A_{G_k}f}) - \alpha J(u_f) \le B + \norm{u_f-f}_{L^2}  \sqrt{B}
\end{align*}
and \cref{eq:rate2-J} follows. The same argument works for the one-dimensional result \cref{eq:rate1d-J} with $B$ now being the right-hand side in \cref{eq:estimate1d}.

Turning to \cref{eq:rate3} we combine \cref{eq:energy-difference} with \cref{eq:ortho} to obtain
\begin{align}\label{eq:energy-difference2}
	-\frac12 \norm[\big]{A_{G_k}f - f}^2_{L^2} \le E(u_{A_{G_k}f};A_{G_k}f) - E(u_f;f) \le \frac{B}{2}.
\end{align}
The first term in this double inequality can be estimated using \cref{eq:interpolationineq} and \cref{thm:rateconvaveraging} yielding
\begin{align*}
	\norm[\big]{A_{G_k}f - f}_{L^2}
		&\le \norm[\big]{A_{G_k}f - f}^{t}_{L^1} \norm[\big]{A_{G_k}f - f}^{1-t}_{L^q} \\
		&\le \left( C_d h J(f) \right)^t \norm[\big]{A_{G_k}f - f}^{1-t}_{L^q} = \mathcal{O}(h^t).
\end{align*}
Observing that $t = \frac{q-2}{2(q-1)} = \frac12 - \frac{q'}{2q}$ proves \cref{eq:rate3}. In dimension $d=1$ \cref{eq:energy-difference2} holds with $B$ being the right-hand side in \cref{eq:estimate1d}. Moreover, we have $\norm{A_{G_k}f - f}^2_{L^2} = \mathcal{O}(h)$ directly from \cref{eq:averaging-error-d1}. Hence \cref{eq:rate3-1d} holds.
\end{proof}
\begin{remark}
Regarding $L^2$ approximation, the rate of $\mathcal{O}\big(h^{\sfrac12}\big)$ for $f \in L^{\infty}(\Omega)$ is optimal in every dimension. We show this by providing a specific datum $f$ and a sequence $(G_k)$ such that $u_{A_{G_k}f}$ converges exactly at the rate $h^{\sfrac12}$. Let $\Omega = (0,1)^d$, $E = \{ x \in \Omega \mid x_d < \sfrac12 \}$ and $f = \mathbf{1}_E$. We know from \cref{thm:fpcrupcr} that $u_f$ is of the form $u_f(x) = a$ for $x \in E$ and $u_f(x) = b$ otherwise, for two constants $a, b$. Moreover, we can choose $\alpha$ such that $a \neq b$. Consider a sequence of Cartesian grids $(G_k)$, each with mesh size $h = h(G_k) = \sfrac{1}{k}$, that is, every $R \in \mathcal{Q}(G_k)$ is a translate of $h\Omega$. If $k$ is odd, then $A_{G_k} u_f$ differs from $u_f$ on exactly $k^{d-1}$ hypercubes and a brief calculation shows
\begin{align*}
    \| u_{A_{G_k} f} - u_f \|^2_{L^2} \ge \| A_{G_k} u_f - u_f \|^2_{L^2} %\\
        %& = \sum_{R \in \mathcal{Q}(G)} \int_R |u_f - A_G u_f|^2 \, dx \\
        %& = \sum_{\substack{R \in \mathcal{Q}(G), \\  R \cap \partial E \neq \emptyset}} \int_R |u_f - A_G u_f|^2 \, dx \\
        %& = n^{d-1} \int_{h\Omega} \Big|\frac{a-b}{2}\Big|^2 dx \\
        = c k^{d-1} h^d = ch
\end{align*}
for a $c>0$. Combining this with \cref{eq:rate2-l2}, for $q=\infty$, shows that $\norm{ u_{A_{G_k}f} - u_f}_{L^2}$ goes to zero exactly at the rate $h^{\sfrac12}$. In fact, this example implies that the rate $h^{\sfrac{1}{2}}$ for the $L^2$ approximation error is optimal also for $f\in L^{\infty}(\Omega)\cap BV(\Omega)$.
\end{remark}
\section{Invariant $\varphi$-minimal sets} \label{sec:phi}

We recall the notion of invariant $\varphi$-minimal subsets of $\mathbb{R}^{n}$ as introduced in \cite{Kruglyak2}. 
\begin{definition}
 A set $M\subset\mathbb{R}^{n}$ is called \emph{invariant $\varphi$-minimal} if for any $a\in\mathbb{R}^{n}$ there exists an element $x_{a}\in M$ such that
 \begin{align} \label{phifunctional}
  \sum^{n}_{i=1}\varphi\left(a_{i}-x_{a,i}\right)\leq\sum^{n}_{i=1}\varphi\left(a_{i}-x_{i}\right)
 \end{align}
holds for all $x\in M$ and all convex functions $\varphi:\mathbb{R}\rightarrow\mathbb{R}$.
\end{definition}
For our purposes the following result, which shows that every invariant $\varphi$-minimal set in addition fulfils a weighted version of inequality \cref{phifunctional}, will be useful.
\begin{theorem} \label{prop:invphilambdaminimal}
 Let $M\subset\mathbb{R}^{n}$ be a bounded, closed and convex set that is invariant $\varphi$-minimal. Then for any $a\in\mathbb{R}^{n}$ and any $w\in\mathbb{R}^{n}_{>0}$ there exists an element $x_{a,w}\in M$ such that
 \begin{align*}
  \sum^{n}_{i=1}w_{i}\varphi\left(\frac{a_{i}-x_{a,w,i}}{w_{i}}\right)\leq\sum^{n}_{i=1}w_{i}\varphi\left(\frac{a_{i}-x_{i}}{w_{i}}\right)
 \end{align*}
holds for all $x\in M$ and all convex functions $\varphi:\mathbb{R}\rightarrow\mathbb{R}$.
\end{theorem}
In order to prove \cref{prop:invphilambdaminimal}, we first recall the notion of special cone property which was introduced in \cite{Kruglyak2}.
By $\left\{e_{i}\right\}_{i=1}^{n}$ we denote the standard basis of $\mathbb{R}^{n}$.
\begin{definition} \label[definition]{def:scp}
 Let $M\subset\mathbb{R}^{n}$ be closed and convex. For $x\in M$, consider all vectors $s=e_{i}-e_{j}$ such that $x+\varepsilon s\in M$ for some $\varepsilon>0$. Denote by $S_{x}$ the set of all such vectors at $x$ and let $K_{x}=\{y\in\mathbb{R}^{n}:y=\sum_{s\in S_{x}}\lambda_{s}s,\lambda_{s}\geq 0\}$ be the convex cone generated by these vectors. The set $ M$ is said to have the \emph{special cone property} if $M\subset x+K_{x}$ for each $x\in M$.
\end{definition}
\begin{remark}
 In \cite{Kruglyak2}, vectors of the type $e_{i}$ and $e_{i}+e_{j}$ are also considered in the definition of the special cone property. These vectors are required for the characterization of invariant $K$-minimal sets, a notion related to invariant $\varphi$-minimal sets.
\end{remark}
The following characterization of invariant $\varphi$-minimal sets was established in \cite[Thm. 3.2, Thm. 4.2]{Kruglyak2}.
\begin{theorem} \label {thm:scp}
 A bounded, closed and convex set $M\subset\mathbb{R}^{n}$ is invariant $\varphi$-minimal if and only if it has the special cone property.
\end{theorem}
The next lemma will turn out to be useful for proving \cref{prop:invphilambdaminimal}.
\begin{lemma} \label[lemma]{lemma:linesegment}
 Given $s=e_{k}-e_{l}$ and $b\in\mathbb{R}^{n}$, consider the closed line segment
 \begin{align*}
  L_{c,d}=\left\{x\in\mathbb{R}^{n} \,\middle|\, x=b+ts,t\in\left[c,d\right]\subset\mathbb{R}\right\}.
 \end{align*}
For every $a\in\mathbb{R}^{n}$ and $w\in\mathbb{R}^n_{>0}$ there exists an element $x_{a,w}\in L_{c,d}$ such that
\begin{align*}
  \sum^{n}_{i=1}w_{i}\varphi\left(\frac{a_{i}-x_{a,w,i}}{w_{i}}\right)\leq\sum^{n}_{i=1}w_{i}\varphi\left(\frac{a_{i}-x_{i}}{w_{i}}\right)
\end{align*}
 holds for all $x\in L_{c,d}$ and every convex function $\varphi:\mathbb{R}\rightarrow\mathbb{R}$.
\end{lemma}
\begin{proof}
 Let $a\in\mathbb{R}^{n}$ and $w\in\mathbb{R}^{n}_{>0}$ be given. Consider an arbitrary convex function $\varphi:\mathbb{R}\rightarrow\mathbb{R}$ and let
 \begin{align*} %\label{psi}
  \psi_{a,w}(x)=\sum^{n}_{i=1}w_{i}\varphi\left(\frac{a_{i}-x_{i}}{w_{i}}\right).
 \end{align*}
 It is straightforward to show, using Jensen's inequality, that $\psi_{a,w}$ is minimized on the line $L=\left\{x\in\mathbb{R}^{n}:x=b+ts,t\in\mathbb{R}\right\}$ by $x_{a,w}=b+t_{a,w}s$ where 
 \begin{align*}
 t_{a,w}=\frac{(a_{k}-b_{k})w_{l}-(a_{l}-b_{l})w_{k}}{w_{k}+w_{l}}.
 \end{align*}
Note that $\varphi$ does not influence $t_{a,w}$ and $x_{a,w}$. On $L$, $\psi_{a,w}$ may be viewed as a function of the parameter $t\in\mathbb{R}$ and as such it is decreasing on $\left(-\infty,t_{a,w}\right]$ and increasing on $\left[t_{a,w},\infty\right)$. Therefore, when restricted to the closed segment $L_{c,d}$ of $L$, $\psi_{a,w}$ has a minimum at $x_{a,w}=b+cs$ if $t_{a,w}<c$, at $x_{a,w}=b+t_{a,w}s$ if $c\leq t_{a,w}\leq d$, and at $x_{a,w}=b+ds$ if $t_{a,w}>d$. In all cases, $x_{a,w}$ does not depend on $\varphi$.
\end{proof}
We are now ready to prove \cref{prop:invphilambdaminimal}. The structure of the proof is similar to the one of \cite[Thm. 3.1]{Kruglyak2}.
\begin{proof}[Proof of \cref{prop:invphilambdaminimal}]
Let $a\in\mathbb{R}^{n}$ and $w\in\mathbb{R}^{n}_{>0}$ be given and fix a convex function $\varphi:\mathbb{R}\rightarrow\mathbb{R}$. As in the proof of \cref{lemma:linesegment}, let
 \begin{align} \label{eq:psi2}
  \psi_{a,w}(x)=\sum^{n}_{i=1}w_{i}\varphi\left(\frac{a_{i}-x_{i}}{w_{i}}\right).
 \end{align}
Since $\psi_{a,w}$ is continuous on $M$, it attains a minimum there. %there exist at least one element in $M$ where $\psi_{a,w}$ attains a minimum.
Let $M_{a,w}$ denote the set of all minimizers of $\psi_{a,w}$ in $M$, i.e.
%\begin{align*}
%  M_{a,w}=\left\{y\in M:\psi_{a,w}(y)=\underset{x\in M}{\min}\psi_{a,w}(x)\right\}.
%\end{align*}
\begin{align*}
	M_{a,w} = \argmin_{x\in M} \psi_{a,w}(x).
\end{align*}
Note that $M_{a,w}$ is a closed and convex subset of $M$. Let $x_{a,w}$ denote the element in $M_{a,w}$ which satisfies
\begin{align} \label{eq:squarefunction}
\sum^{n}_{i=1}w_{i}\left(\frac{a_{i}-x_{a,w,i}}{w_{i}}\right)^{2}=\min_{x\in M_{a,w}}\sum^{n}_{i=1}w_{i}\left(\frac{a_{i}-x_{i}}{w_{i}}\right)^{2}.
\end{align}
Uniqueness of $x_{a,w}$ follows from the strict convexity of $x \mapsto \sum^{n}_{i=1}w_{i}\big(\frac{a_{i}-x_{i}}{w_{i}}\big)^{2}$. Equivalently, \cref{eq:squarefunction} can be formulated as
\begin{align*}
\lVert a-x_{a,w}\rVert_{2,1/w}=\underset{x\in M_{a,w}}{\min}\lVert a-x\rVert_{2,1/w}
\end{align*}
where
\begin{align*}
\lVert x\rVert_{2,1/w}=\sqrt{\langle x,x\rangle_{1/w}}
\end{align*}
is the norm associated to the weighted inner product on $\mathbb{R}^{n}$ given by
\begin{align*}
\langle x,y\rangle_{1/w}=\sum^{n}_{i=1}x_{i}y_{i}\frac{1}{w_{i}}.
\end{align*}
The remaining part of the proof is devoted to showing that $x_{a,w}$ is the element of best approximation in $M$ of $a$ with respect to $\lVert\cdot\rVert_{2,1/w}$. As this element is unique, it then follows that $x_{a,w}$ is independent of the specific $\varphi$ and minimizes \cref{eq:psi2} in $M$ for all convex functions.
\par Now, $M$ is a set in the class characterized by \cref{thm:scp} and has therefore the special cone property. Take a vector $s\in S_{x_{a,w}}$ and recall that $x_{a,w}+\varepsilon s\in M$ for some $\varepsilon>0$. The convexity of $M$ gives that the entire line segment $L_{0,\varepsilon}=\left\{x:x=x_{a,w}+ts,t\in\left[0,\varepsilon\right]\right\}$ is in $M$. By construction, $x_{a,w}$ minimizes \cref{eq:psi2} on $L_{0,\varepsilon}$. Further, if there are several minimizers of \cref{eq:psi2} on $L_{0,\varepsilon}$, $x_{a,w}$ has the smallest $\lVert\cdot\rVert_{2,1/w}$-norm among them. Taking into account these two properties of $x_{a,w}$, \cref{lemma:linesegment} gives that $x_{a,w}$ minimizes $\lVert\cdot\rVert_{2,1/w}$ on $L_{0,\varepsilon}$. We can then derive
\begin{align*}
\lVert a-x_{a,w}\rVert^{2}_{2,1/w}&\leq\lVert a-(x_{a,w}+ts)\rVert^{2}_{2,1/w}\\
&=\lVert a-x_{a,w}\rVert^{2}_{2,1/w}-2t\langle a-x_{a,w},s\rangle_{1/w}+t^{2}\langle s,s\rangle_{1/w}
\end{align*}
for $t\in\left[0,\varepsilon\right]$. Therefore, for $t$ in this interval,
\begin{align*}
 -2t\langle a-x_{a,w},s\rangle_{1/w}+t^{2}\langle s,s\rangle_{1/w}\geq 0
\end{align*}
which in turn gives that
\begin{align} \label{estinner}
 \langle a-x_{a,w},s\rangle_{1/w}\leq 0.
\end{align}
\par Consider now a general element $x\in x_{a,w}+K_{x_{a,w}}$, i.e.\ $x=x_{a,w}+\sum_{s\in S_{x_{a,w}}}\lambda_{s}s$ where $\lambda_{s}\geq 0$ (recall \cref{def:scp}). We have
\begin{align*}
	\norm{a-x}^{2}_{2,1/w}
		&=	\norm[\Big]{ a- \Big( x_{a,w} + \sum_{\mathclap{s\in S_{x_{a,w}}}}\lambda_{s}s \Big)}^{2}_{2,1/w}\\
 		&=	\norm{a-x_{a,w}}^{2}_{2,1/w} - 2\sum_{\mathclap{s\in S_{x_{a,w}}}}\lambda_{s}\langle a-x_{a,w},s\rangle_{1/w} + \sum_{\mathclap{s,r\in S_{x_{a,w}}}}\lambda_{s}\lambda_{r}\langle s,r\rangle_{1/w}.
\end{align*}
The inequality \cref{estinner} implies that
\begin{align} \label{estinner2}
	\sum_{\mathclap{s\in S_{x_{a,w}}}}\lambda_{s}\langle a-x_{a,w},s\rangle_{1/w}\leq 0.
\end{align}
Further, 
\begin{align} \label{estinner3}
	\sum_{\mathclap{s,r\in S_{x_{a,w}}}}\lambda_{s}\lambda_{r}\langle s,r\rangle_{1/w}\geq 0
\end{align}
as $\left(\langle s,r\rangle_{1/w}\right)\in\mathbb{R}^{|S_{x_{a,w}}|\times |S_{x_{a,w}}|}$ is a Gram matrix and therefore positive semidefinite. From \cref{estinner2} and \cref{estinner3}, we get the estimate
\begin{align} \label{eq:l2w}
 \lVert a-x_{a,w}\rVert_{2,1/w}\leq\lVert a-x\rVert_{2,1/w}.
\end{align}
In view of \cref{eq:l2w}, which holds for any $x\in x_{a,w}+K_{x_{a,w}}$, and $M\subset x_{a,w}+K_{x_{a,w}}$, the element $x_{a,w}$ turns out to be the unique best approximation in $M$ of $a$ with respect to $\lVert\cdot\rVert_{2,1/w}$. As $\varphi$ was arbitrarily chosen, $x_{a,w}$ minimizes \cref{eq:psi2} in $M$ given any choice of convex function.
\end{proof}

\section{The ROF model on weighted graphs} \label{sec:rof}

In this section we introduce a discrete formulation of the ROF model that includes, as special cases, the finite-dimensional instances of \cref{eq:ROF} for $f \in PCR_G$; recall \cref{rem:uAf}. If $G$ partitions $\Omega$ into hypercubes, the minimizer enjoys a discrete universal minimality property according to \cite[Thm.\ 3.2]{KirSchSet19a}. \Cref{thm:rofinvariancegraph} below implies that this statement extends to arbitrary rectilinear grids.

Similar ROF-type energies on graphs have been considered, for instance, in \cite{BouElmMel09,ChaOshShe01,CouGraNajPesTal13,GraPol10,PeyBouCoh08,ZhoScho05}. In these works \enquote{isotropic} definitions of the total variation are prevalent, where local variation is measured by the Euclidean norm of the gradient. Note, however, that the subdifferentials of such discrete total variation functionals are generally not invariant $\varphi$-minimal \cite{KirSchSet19a}. On the other hand, the use of the $\ell^1$ norm instead of the Euclidean one is common in graph cuts approaches to the ROF problem, see e.g.\ \cite{Cha05,ChaDar09,DarSig06,GolYin09,Hoc01}.

Let $(\sV,\sE)$ be a finite, oriented, connected graph. That is, (i) $\sV = \{\mathsf{v_1},\ldots,\mathsf{v_n}\}$ and $\sE\subset \sV\times \sV$, (ii) if $(\sv_\si,\sv_\sj)\in \sE$, then $(\sv_\sj,\sv_\si) \notin \sE$ must hold, and (iii) for every pair of vertices there is path that connects them ignoring edge orientation.

Denote by $\mathbb{R}^{\sV}$ and $\mathbb{R}^{\sE}$ the spaces of all real-valued functions defined on $\sV$ and $\sE$, respectively. We endow both $\sV$ and $\sE$ with positive weight functions $\sw\in\mathbb{R}^{\mathsf{V}}_{>0}$ and $\sW\in\mathbb{R}^{\sE}_{>0}$. We also endow $\mathbb{R}^{\sV}$ with the following weighted inner product and weighted $\ell^p$ norms
\begin{align*}
	\langle \mathsf{g}, \mathsf{h} \rangle_{\sw} &= \sum_{\mathsf{v \in V}} \mathsf{w(v) g(v) h(v)}, \\
	\lVert \mathsf{g}\rVert^p_{p,\sw} &= \sum_{\sv \in \sV}\sw(\sv)\left|\mathsf{g}(\sv)\right|^{p}, \quad 1\leq p<\infty.
\end{align*}
\begin{definition}\label[definition]{def:wdiv}
The \emph{divergence} $\sfdivergence:\mathbb{R}^{\sE} \to \mathbb{R}^{\sV}$ on the weighted graph is given by
\begin{equation*}
	\sfdivergence\sH(\sv) = \frac{1}{\sw(\sv)} \Bigg( \smashoperator[r]{\sum_{\se=(\bar{\sv},\sv)\in \sE}}\sW(\se)\sH(\se)-\sum_{\mathclap{\se=(\sv,\bar{\sv})\in \sE}}\sW(\se)\sH(\se) \Bigg), \quad \sH\in \mathbb{R}^\sE.
\end{equation*}
\end{definition}
In analogy to \cref{eq:B} we define $\sB = \{\sH  \in \mathbb{R}^{\sE} \mid  -1 \le \sH \le 1 \}$.
\begin{definition}\label[definition]{def:wtv}
The \emph{total variation} $\sJ:\mathbb{R}^\sV \to \mathbb{R}$ on the weighted graph is defined as
\begin{equation*}
	\sJ(\su) = \sigma_{\sfdivergence(\sB)}(\su) = \sup_{\mathsf{h}\in\sfdivergence(\sB)} \langle \su, \mathsf{h} \rangle_{\sw}
	%\sum_{\sV\in\sV} \sw(\sV) \su(\sV) \mathsf{h}(\sV)
\end{equation*}
where $\sfdivergence(\mathsf{B})$ denotes the image of $\mathsf{B}$ under the weighted divergence operator.
\end{definition}
The next lemma shows that $\sJ$ is independent of both vertex weights and edge orientation.
\begin{lemma}\label[lemma]{thm:Jnosup}
For every $\su \in \mathbb{R}^{\sV}$
\begin{align}
	\sJ(\su)
		% &= \sup_{\mathsf{h}\in\sfdivergence_{\sW,1} \mathsf{B}_1}\sum_{\sV\in \sV}\su(\sV)\mathsf{h}(\sV) \label{eq:Jnosup1} \\
		&= \sum_{\mathclap{(\sv_\si,\sv_\sj)\in \sE}} \sW(\sv_\si,\sv_\sj)|\su(\sv_\si) - \su(\sv_\sj)|. \label{eq:Jnosup2}
\end{align}
\end{lemma}
\begin{proof}
We use \cref{def:wdiv,def:wtv} and rewrite 
\begin{align}\label{eq:STV}
	\sJ(\su)
		&=	\sup_{\sH\in\mathsf{B}} \sum_{\sv \in \sV} \sw(\sv) \su(\sv) (\sfdivergence\sH)(\sv) \notag \\
 		&=	\sup_{\sH\in\mathsf{B}} \smashoperator[r]{\sum_{(\sv_\si,\sv_\sj)\in \sE}} \sW(\sv_\si,\sv_\sj) (\su(\sv_\sj)-\su(\sv_\si))\sH(\sv_\si,\sv_\sj).
\end{align}
The supremum in \cref{eq:STV} is attained for $\sH\in\mathsf{B}$ if and only if
\begin{align*}
	\sH(\mathsf{v_{i}},\mathsf{v_{j}}) \in
		\left\{
		\begin{array}{lr}
			\left\{1\right\},	& \su(\sv_\si)<\su(\sv_\sj),\\
			\left[-1,1\right],	& \su(\sv_\si)=\su(\sv_\sj),\\
			\left\{-1\right\},	& \su(\sv_\si)>\su(\sv_\sj),
		\end{array}
		\right.
\end{align*}
giving \cref{eq:Jnosup2}.
\end{proof}
The analogue of problem \cref{eq:ROF} on the weighted graph is
\begin{align} \label{eq:ROFgraph}
	\minimize_{\su\in\mathbb{R}^{\sV}} \quad  \frac{1}{2} \left\| \sff-\su \right\|^{2}_{2,\sw} + \alpha \mathsf{J}(\su), \tag{P$'$}
\end{align}
where $\sff\in\mathbb{R}^{\sV}$ and $\alpha>0$ are given. Problem \cref{eq:ROFgraph} is also well-posed due to the results given in \cref{sec:convex} and we set
\begin{align*}
	\mathsf{u_f} = \argmin_{\su \in \mathbb{R}^{\sV}}  \left\{ \frac{1}{2} \left\| \sff-\su \right\|^{2}_{2,\sw} + \alpha \mathsf{J}(\su) \right\}.
\end{align*}
Moreover, we know from \cref{thm:support-fct,thm:dual2} that 
\begin{align}\label{eq:graph-dual}
	\su_\sff = \argmin_{\su\in \sff - \alpha \partial \sJ(\mathsf{0})} \| \su \|_{2,\sw}.
\end{align}
Note that the subdifferential is taken with respect to the weighted inner product $\langle \cdot , \cdot \rangle_{\sw}$.
\begin{theorem}\label{thm:rofinvariancegraph}
For every convex function $\varphi:\mathbb{R}\rightarrow\mathbb{R}$
\begin{align*}
%	\sum_{\sv \in \sv}\varphi\left(\su_{\sff}(\sv)\right)\sw(\sv) = \min_{\su\in \sff-\alpha \partial \mathsf{J_W}(\mathsf{0})} \sum_{\sv \in \sV}\varphi\left(\su(\sv)\right)\sw(\sv).
	\su_{\sff} \in \argmin_{\su \in \sff-\alpha \partial \sJ(\mathsf{0})} \sum_{\sv \in \sV}\varphi\left(\su(\sv)\right)\sw(\sv).
\end{align*}
\end{theorem}
\begin{proof}
From \cite[Thm.~2.4, Rem.~2.5]{Kruglyak3} it follows that the bounded, closed and convex set $\alpha\sw\sfdivergence(\mathsf{B}) = \alpha \sw \partial \mathsf{J}(\mathsf{0})$ is invariant $\varphi$-minimal. \Cref{prop:invphilambdaminimal} guarantees the existence of a unique element $\hat{\mathsf{x}}\in\alpha \sf{w} \partial \mathsf{J}(\mathsf{0})$ satisfying
\begin{align*} %\label{eq:invphiweight}
%	\sum_{\sv \in \sV} \sw(\sv) \varphi \left(\frac{\sw(\sv)\sff(\sv)-\hat{\mathsf{x}}(\sv)}{\sw(\sv)}\right) =
%	\min_{\mathsf{x}\in\alpha \sw \partial \mathsf{J_W}(\mathsf{0})}\sum_{\sv \in \sV}\sw(\sv)\varphi\left(\frac{\sw(\sv)\sff(\sv)-\mathsf{x}(\sv)}{\sw(\sv)}\right)
	\hat{\mathsf{x}} \in \argmin_{\mathsf{x}\in\alpha \sw \partial \mathsf{J}(\mathsf{0})}\sum_{\sv \in \sV}\sw(\sv)\varphi\left(\frac{\sw(\sv)\sff(\sv)-\mathsf{x}(\sv)}{\sw(\sv)}\right)
\end{align*}
for every convex function $\varphi:\mathbb{R}\rightarrow\mathbb{R}$. Equivalently, $\hat{\mathsf{y}} = \hat{\mathsf{x}}/\sw$ is the unique element in $\alpha \partial \mathsf{J}(\mathsf{0})$ such that
\begin{align} \label{eq:invphiweight}
%	\sum_{\sv \in \sV} \sw(\sv) \varphi \left(\sff(\sv)-\hat{\mathsf{y}}(\sv)\right) =
%	\min_{\mathsf{y}\in\alpha \partial \mathsf{J_W}(\mathsf{0})}\sum_{\sv \in \sV}\sw(\sv)\varphi\left(\sff(\sv)-\mathsf{y}(\sv)\right)
	\hat{\mathsf{y}} \in \argmin_{\mathsf{y}\in\alpha \partial \mathsf{J}(\mathsf{0})}\sum_{\sv \in \sV}\sw(\sv)\varphi\left(\sff(\sv)-\mathsf{y}(\sv)\right)
\end{align}
for every convex function $\varphi:\mathbb{R}\rightarrow\mathbb{R}$. Comparing \cref{eq:invphiweight}, for $\varphi(t)=t^{2}$, with \cref{eq:graph-dual} then gives $\sff-\hat{\mathsf{y}}=\su_{\sff}$ and the result follows.
\end{proof}
\begin{remark}
\Cref{thm:rofinvariancegraph} implies in particular that $\su_{\sff}$ minimizes $\|\cdot\|_{p,\sw}$ for all $p \in [1,\infty)$ over the set $\sff-\alpha \partial \mathsf{J}(\mathsf{0})$. Observing that $\|\su\|_{p,\sw} \to \|\su\|_{\infty}$ as $p\to \infty$, it follows that $\su_{\sff}$ also minimizes $\|\cdot\|_{\infty}$ over this set.
\end{remark}
\section{Universal minimality of the ROF minimizer}\label{sec:main}
In order to translate \cref{thm:rofinvariancegraph} to the continuous setting we first construct a weighted graph so that problem \cref{eq:ROFgraph} is equivalent to \cref{eq:ROF} whenever $f \in PCR_G$.
\begin{definition}\label[definition]{def:graph}
Given a hyperrectangle $\Omega \subset \mathbb{R}^d$ and a grid $G$ on $\Omega$ we define the graph $(\sV_G,\sE_G)$ by identifying $\sV_G$ with the partition $\mathcal{Q}(G)$ and $\sE_G$ with
\begin{equation}\label{eq:sides}
	\mathcal{S}(G) = \left\{S_{ij} = \overline{R_i} \cap \overline{R_j} \, \middle| \, R_i,R_j \in \mathcal{Q}(G), H^{d-1}(S_{ij})>0 \right\}.
\end{equation}
That is, there is a vertex $\sv_\si$ for every hyperrectangle $R_i$ as well as one arbitrarily oriented edge $(\sv_\si,\sv_\sj)$ for every pair of hyperrectangles $(R_i,R_j)$ sharing a joint side. The weight functions $\sw\in \mathbb{R}^{\sV_G}_{>0}$ and $\sW\in \mathbb{R}^{\sE_G}_{>0}$ associated to the graph $(\sV_G,\sE_G)$ are defined by
\begin{equation}\label{eq:weights}
	\begin{aligned}
		\sw(\sv_\si) &= |R_i|, \\
		\sW(\sv_\si,\sv_\sj) &= H^{d-1}(S_{ij}).
	\end{aligned}
\end{equation}
\end{definition}
With this construction at hand we identify the two spaces $PCR_G$ and $\mathbb{R}^{\sV_G}$ by means of the following isomorphism
\begin{align*}
	\iota	&: PCR_G \to \mathbb{R}^{\sV_G}, \quad \iota(g)(\sv_\si) = g_i.
\end{align*}
Note that
\begin{gather*}
	\int_\Omega f(x) g(x) \, dx = \langle \iota(f), \iota(g) \rangle_{\sw}, \quad \text{and}\\
	J(f) = \sJ(\iota(f))
\end{gather*}
for all $f,g \in PCR_G$. It now follows from \cref{thm:fpcrupcr} that
\begin{align}\label{eq:equivalence}
	\iota(u_f) = \su_{\iota(f)}
\end{align}
for all $f \in PCR_G$.

\begin{remark}
For the sake of completeness we mention that the space $\mathbb{R}^{\sE_G}$ also has a natural counterpart in the continuous setting: the space $PAR_G^0$ of lowest order, i.e.\ piecewise affine, Raviart-Thomas vector fields on the partition $\mathcal{Q}(G)$ with vanishing normal trace on $\partial \Omega$. Introduced in \cite{RavTho77}, Raviart-Thomas vector fields are commonly used in conforming finite element approximations of the space $H(\operatorname{div})$ of $L^2$ vector fields with $L^2$ divergences, see also \cite{BofBreFor13}. The elements of $PAR_G^0$ have a well-defined and constant normal trace on each joint face $S_{ij} \in \mathcal{S}(G)$, recall \cref{eq:sides}. Thus, we can define an isomorphism $\kappa : PAR^0_G  \to \mathbb{R}^{\sE_G}$ by letting $\kappa(F)(\sv_\si,\sv_\sj)$ equal the normal trace of $F$ on $S_{ij}$, up to a potential change of sign depending on the orientation of the edge $(\sv_\si,\sv_\sj)$. With this definition the relation $\iota \circ \divergence = \sfdivergence \circ \kappa$ holds on $PAR_G^0$. See also \cref{fig:cd}.
\end{remark}

\begin{figure}
\centering
\begin{tikzpicture}
  \matrix (m) [matrix of math nodes,row sep=3em,column sep=4em,minimum width=2em]
  {
     PCR_G		& \mathbb{R}^{\sV_G} \\
     PAR^0_G	& \mathbb{R}^{\sE_G} \\};
  \path[-stealth]
    (m-2-1) edge node [left] {$\divergence$} (m-1-1)
    (m-1-1) edge node [above] {$\iota$} (m-1-2)
    (m-2-1.east|-m-2-2) edge node [below] {$\kappa$} (m-2-2)
    (m-2-2) edge node [right] {$\sfdivergence$} (m-1-2);
\end{tikzpicture}
\caption{Commutative diagram illustrating the relations between the four spaces $PCR_G$, $PAR^0_G$, $\mathbb{R}^{\sV_G}$ and $\mathbb{R}^{\sE_G}$.} \label{fig:cd}
\end{figure}

\begin{lemma} \label[lemma]{thm:AdJ}$\iota (A_G (\partial J(0))) = \partial \mathsf{J}(0).$
\end{lemma}
\begin{proof}
Let $p = A_G q$, $q \in \partial J(0).$ Then $\langle p, g \rangle = \langle q, g \rangle \le J(g)$ for all $g \in PCR_G$. Translating to the graph, we have $\langle \iota(p), \mathsf{g} \rangle_{\mathsf{w}} \le \mathsf{J(g)}$ for all $\mathsf{g} \in \mathbb{R}^{\mathsf{V}}$. Hence, $\iota (p) \in \partial \mathsf{J}(0).$

Conversely, let $\mathsf{p} \in \partial \mathsf{J}(0).$ Then $\langle \mathsf{p}, \mathsf{g} \rangle_{\mathsf{w}} \le \mathsf{J(g)}$ for all $\mathsf{g} \in \mathbb{R}^{\mathsf{V}}$. Equivalently, $\langle \iota^{-1}(\mathsf{p}), A_G u \rangle \le J(A_G u)$ for all $u \in L^2(\Omega)$. From the selfadjointness and the total variation diminishing property of $A_G$ we obtain $\langle \iota^{-1}(\mathsf{p}), u \rangle \le J(u)$. Thus, $\iota^{-1}(\mathsf{p}) \in \partial J(0)$. Since $A_G \iota^{-1}(\mathsf{p}) = \iota^{-1}(\mathsf{p})$, we also have $\iota^{-1}(\mathsf{p}) \in A_G( \partial J(0))$.
\end{proof}
We can now prove a special case of \cref{thm:invgeneral}.
\begin{lemma} \label[lemma]{thm:invpcr}
Given datum $f\in PCR_{G}$, the ROF minimizer $u_f$ satisfies
\begin{align*}
%	\int_{\Omega}\varphi(u_f(x))dx = \min \left\{ \int_{\Omega}\varphi(u(x))dx \, \middle| \, u\in f-\alpha A_{G}(\partial J(0)) \right\}
	\int_{\Omega}\varphi(u_f(x)) \, dx \le \int_{\Omega}\varphi(u(x)) \, dx
\end{align*}
for all $u\in f-\alpha A_{G}(\partial J(0))$ and every convex function $\varphi:\mathbb{R}\rightarrow\mathbb{R}$.
\end{lemma}
\begin{proof}
By \cref{eq:equivalence} $\iota(u_f)$ solves \cref{eq:ROFgraph} for datum $\iota(f)$. Picking an arbitrary convex $\varphi:\mathbb{R}\rightarrow\mathbb{R}$ and applying \cref{thm:rofinvariancegraph} we obtain
\begin{align*}
	\sum_{\sv\in \sV} \varphi\left(\iota(u_f)(\sv)\right)\sw(\sv) \le \sum_{\sv \in \sV }\varphi\left(\su(\sv)\right)\sw(\sv)
\end{align*}
for all $\su \in \iota(f) - \alpha \partial \sJ(0)$. Equivalently,
\begin{align*}
	\int_{\Omega}\varphi(u_f (x))\,dx \le \int_{\Omega}\varphi(u(x))\,dx
\end{align*}
for all $u \in \iota^{-1}(\iota(f) - \alpha \partial \sJ(0)) = f - \alpha A_G (\partial J(0))$, because of \cref{thm:AdJ}.
\end{proof}

\begin{theorem} \label{thm:invgeneral}
 Given datum $f\in L^{2}(\Omega)$, the ROF minimizer $u_f$ satisfies
\begin{align*}
	\int_{\Omega}\varphi(u_f(x))dx \le \int_{\Omega}\varphi(u(x))dx
\end{align*}
for every $u \in f-\alpha\partial J(0)$ and every convex function $\varphi:\mathbb{R}\rightarrow\mathbb{R}$.
\end{theorem}
\begin{proof}
Let $(G_k)$ be a sequence of grids such that $h(G_k) \to 0$. Due to \cref{thm:fpcrupcr} $u_k \coloneqq u_{A_{G_k}f} \in PCR_{G_k}$ and, due to \cref{thm:rate}, $u_k \to u_f$ in $L^{2}(\Omega)$.

Consider an arbitrary element $u\in f-\alpha\partial J(0)$. It is clear that $A_{G_{k}}u\in A_{G_{k}}f-\alpha A_{G_{k}}(\partial J(0))$.
Moreover, we also have $u_{k}\in A_{G_{k}}f-\alpha A_{G_{k}}(\partial J(0))$. Applying first \cref{thm:invpcr} and then \cref{lemma:operatorA}, we obtain the inequalities
\begin{align*}
	\int_{\Omega}\varphi\left(u_{k}(x)\right)dx\leq\int_{\Omega}\varphi\left((A_{G_{k}}u)(x)\right)dx\leq\int_{\Omega}\varphi\left(u(x)\right)dx,
\end{align*}
which are valid for every convex $\varphi:\mathbb{R} \to \mathbb{R}$.

The remaining step is to verify lower semicontinuity of $g \mapsto \int_{\Omega}\varphi\left(g(x)\right)dx$ in $L^{2}(\Omega)$. This, however, is a direct consequence of the convexity of $\varphi$. See, for instance, \cite[Thm. 3.20]{Dacorogna1}.
Therefore
\begin{align*}
	\int_{\Omega}\varphi\left(u_f(x)\right)dx\leq\underset{k\rightarrow\infty}{\liminf}\int_{\Omega}\varphi\left(u_{k}(x)\right)dx\leq\int_{\Omega}\varphi\left(u(x)\right)dx.
\end{align*}
\end{proof}
\begin{remark}\label[remark]{rem:prox-nonexpansive}
By choosing $\varphi(\cdot) = \left|\cdot\right|^p$, $1 \le p < \infty$, Theorem \ref{thm:invgeneral} implies that $\left\| u_f \right\|_{L^p}  \le \left\| u \right\|_{L^p}$ for all $u \in f - \alpha \partial J(0)$. A limiting argument shows that the inequality also holds for $p=\infty.$ Note, however, that the norms might be infinite for $p>2.$ %It follows that the ROF model preserves integrability of the datum. 
In particular, we have
\begin{align}\label{eq:rof-nonexpansive}
	\norm{u_f}_{L^p(\Omega)} \le \norm{f}_{L^p(\Omega)}
\end{align}
for all $p\in [1,\infty],$ showing finiteness of the error bounds in \eqref{eq:estimate2} and \eqref{eq:estimate1d}.
\end{remark}

\begin{remark}
Theorem \ref{thm:invgeneral} is analogous to \cite[Thm.\ 4.46]{Scherzer1}, which concerns the isotropic ROF model. While it might be possible to adapt the arguments of \cite[Thm.\ 4.46]{Scherzer1} to the anisotropic setting, we think the proof of Theorem \ref{thm:invgeneral} is of independent interest, revealing precise connections between the discrete and continuous settings. On the other hand, since the essential Theorem \ref{thm:fpcrupcr} fails to hold for the isotropic ROF model (see \cite[Prop.\ 6]{Meyer1}), a proof based on the spaces $PCR_G$ cannot be directly applied to that setting. We also mention that the case $p=\infty$ of inequality \eqref{eq:rof-nonexpansive} follows from a maximum principle. See the proof of \cite[Lem.~2.1]{Cha04} or \cite[Prop.~3.6]{IglMer21}.
\end{remark}

\begin{remark}
In \cite{Kruglyak3} invariant $\varphi$-minimal sets were introduced as subsets of $L^1(\Omega).$ We can extend this definition to $L^p(\Omega)$, $1\le p \le \infty$, in the following way. A set $M \subset L^p(\Omega)$ is \emph{invariant $\varphi$-minimal}, if for every $g \in L^p(\Omega)$ there is an element $u^*_g \in M$ of best approximation in the sense that
\begin{equation*}
	\int_\Omega \varphi(u^*_g - g)\,dx \le \int_\Omega \varphi(u - g)\,dx
\end{equation*}
holds for all $u \in M$ and all convex functions $\varphi.$ This property is invariant under translation and scaling, meaning that, for every $\beta \in \mathbb{R}$ and $h\in L^p(\Omega)$, the set $h+ \beta M$ is again invariant $\varphi$-minimal. Theorem \ref{thm:invgeneral} implies that $\alpha \partial J(0)$ is an invariant $\varphi$-minimal set in $L^2(\Omega)$ for every $\alpha >0$. The scaled subdifferential of the isotropic total variation $\alpha \partial J_2(0)$ is another example of an invariant $\varphi$-minimal set in $L^2(\Omega)$ by \cite[Thm.\ 4.46]{Scherzer1}. Both results provide a characterization of the element $u^*_{f} \in \alpha \partial J(0)$ of best approximation to $f\in L^2(\Omega)$ in terms of the respective ROF minimizer
$$ u^*_{f} = f - u_f.$$
Note that $u^*_f$ is nothing but the solution of the Fenchel dual of the ROF problem.
\end{remark}

\subsection*{Acknowledgments}
We acknowledge support by the Austrian Science Fund (FWF) within the national research network ``Geometry $+$ Simulation," S117, subproject 4. The financial support by the Austrian Federal Ministry for Digital and Economic Affairs, the National Foundation for Research, Technology and Development and the Christian Doppler Research Association is gratefully acknowledged. This work was initiated while ES was affiliated with the Johann Radon Institute for Computational and Applied Mathematics (RICAM) of the Austrian Academy of Sciences. For open access purposes, the authors have applied a CC BY public copyright license to any author-accepted manuscript version arising from this submission.

\end{document}